\documentclass[11pt,reqno]{amsart}
\usepackage[left=1.3in, right=1.2in, top=1in, bottom=1in, includefoot]{geometry}
\usepackage{amssymb,amsmath}
\usepackage{bm}
\usepackage{mathrsfs}
\usepackage{colonequals}
\usepackage[english, french]{babel}
\usepackage{hyperref}
\usepackage[all]{xy}
\usepackage{tikz}
\usepackage{cjhebrew}

\numberwithin{equation}{section}

\newtheorem{theorem}{Theorem}[section]

\newtheorem{lemma}[theorem]{Lemma}

\newtheorem{corollary}[theorem]{Corollary}
\theoremstyle{remark}

\theoremstyle{definition}
\newtheorem{definition}[theorem]{Definition}
\newtheorem{remark}[theorem]{Remark}

\def\C{{\mathbb{C}}}
\def\R{{\mathbb{R}}}
\def\T{{\mathbb{T}}}
\def\Tcal{{\mathcal{T}}}

\def\Mcal{{\mathcal{M}}}
\def\Fcal{{\mathcal{F}}}

\def\Tfrak{{\mathfrak{T}}}

\newcommand{\jb}[1]{\langle #1 \rangle}
\renewcommand{\hat}[1]{\widehat{#1}}
\newcommand{\bxi}{\overline{\xi}}
\newcommand{\bmxi}{\bm{\xi}} 
\newcommand{\ind}{\mathbf{1}} 

\newcommand{\noi}{\noindent}

\newcommand{\Qcal}{\mathcal{Q}}

\newcommand{\TT}{\textup{T}}
\newcommand{\TF}{\mathfrak{T}}
\newcommand{\SF}{\mathfrak{S}}

\newcommand{\F}{\mathcal{F}}

\newcommand{\wt}{\widetilde}
\newcommand{\cj}{\overline}
\newcommand{\dx}{\partial_x}
\newcommand{\dt}{\partial_t}

\title[Unconditional Uniqueness for DNLS on $\R$]{
Unconditional uniqueness for the derivative nonlinear Schr\"{o}dinger equation 
on the real line}

\author[R. Mosincat]{Razvan Mosincat}
\email{r.o.mosincat@sms.ed.ac.uk}
\address{
Maxwell Institute for Mathematical Sciences\\
School of Mathematics\\
University of Edinburgh\\ 
Edinburgh\\
UK\\
EH9 3FD}

\author[H. Yoon]{Haewon Yoon}
\email{hwyoon@ncts.ntu.edu.tw}
\address{National Center for Theoretical Sciences, No. 1 Sec. 4 Roosevelt Rd., National Taiwan University, Taipei, 10617, Taiwan}

\date{\today}



\begin{document}

\selectlanguage{english}

\maketitle

\begin{abstract}

We prove the unconditional uniqueness of solutions to 
the derivative nonlinear Schr\"odinger equation (DNLS) in an almost end-point regularity.  
To this purpose, 
we employ the normal form method and we transform (a gauge-equivalent) DNLS 
into a new equation (the so-called normal form equation) 
for which nonlinear estimates can be easily established in $H^s(\R)$, $s>\frac12$, 
without appealing to an auxiliary function space. 
Also, we prove that low-regularity solutions of DNLS 
satisfy the normal form equation 
and this is done by means of estimates in the $H^{s-1}(\R)$-norm.  
\end{abstract}


\section{Introduction}

We consider the initial-value problem for the derivative nonlinear
Schr\"odinger equation (DNLS) on the real line, i.e.
\begin{equation}
  \label{DNLS}
  \begin{cases}
    i \partial_t u + \partial_x^2 u = i\partial_x (|u|^2u)\\
    u|_{t=0}=u_0\in H^s(\R),
  \end{cases}
  \quad(t,x)\in\R\times\R,
\end{equation}
where $u$ is a complex-valued unknown.  This PDE arises as a model
equation in plasma physics, see e.g. \cite{Rogister71, Mjolhus76}.
Moreover, since it is completely integrable \cite{KaupNewell} it has a
rich structure (e.g. infinitely many conservation laws).  From the
analytical point of view, it poses interesting technical challenges
due to the presence of the derivative in the nonlinear cubic term in
the context of Schr\"{o}dinger dispersion.

The initial-value problem \eqref{DNLS} has been intensely studied both
for smooth, high-regularity (say, $s\ge 1$) initial data
\cite{Lee89,HayashiOzawaSIAM,JLPS,PelinShimab} as well as for
low-regularity initial data \cite{Tak99, CKSTT, CKSTTrefined,
  MiaoWuXu, GuoWu,Pornnopparath}. 
  For the discussion of this section, it is
relevant to recall the result of \cite{Tak99}: by using the Fourier
restriction norm method (i.e. using $X^{s,b}$ spaces) and a gauge
transformation (see e.g. \cite{HayashiOzawaSIAM}), Takaoka showed that
DNLS is locally well-posed in $H^s(\R)$, for $s\geq \frac12$.
However, the uniqueness of solutions holds \emph{conditionally}: for
any $u_0\in H^s(\R)$, there exist $T>0$ and a unique solution
$u\in C([-T,T]; H^s(\R))\cap X_T$ to \eqref{DNLS}, where $X_T$ is some
auxiliary function space.  In other words, for given initial data, the
solution is guaranteed to be unique \emph{only} in the subspace
$ C([-T,T]; H^s(\R))\cap X_T$.

\subsection{Main result}

In this paper, we study the uniqueness of low-regularity solutions to
DNLS.  In particular, we are preoccupied to establish the
\emph{unconditional uniqueness} of solutions to \eqref{DNLS} in
$H^s(\R)$, for $s<1$.

Generally speaking, provided that we can make sense of the
nonlinearity (as a distribution) without assuming that the solution
belongs to some auxiliary function space $X_T$, we establish the
\emph{unconditional well-posedness} for a given PDE by removing the
auxiliary function space from the uniqueness statement of its
well-posedness theory.

Our main result is the following:

\begin{theorem}
  \label{mainthm}
  Let $s>\frac12$.  Then, DNLS is unconditionally (locally) well-posed
  in $H^s(\R)$.
\end{theorem}

The \emph{unconditional well-posedness} is a notion of well-posedness
that does not depend on how the solutions were constructed.  It was
Kato \cite{Kato95} who first studied the issue of whether or not one
can remove an auxiliary function space from the well-posedness
statement for the nonlinear Schr\"{o}dinger equation and thus
strengthen its uniqueness property.  Since then, the uniqueness of
solutions for various other nonlinear dispersive PDEs was
investigated -- see
e.g. \cite{ChristNonuniqueness, FurPlanchonTer, GKO,  KishimotoProc, KOY, Zhou}.

The proof of Theorem~\ref{mainthm} is based on the normal form
approach to unconditional well-posedness of Kwon, Oh, and Yoon
\cite{KOY}, where it was developed an infinite iteration scheme of
normal form reductions in an abstract form for nonlinear dispersive
PDEs on the real line.  This approach builds upon previous works
\cite{GKO,KwOh11} where the normal form method was applied to PDEs
with periodic boundary conditions.  In addition, we also rely on the
abstract variation of the normal form method 
due to Kishimoto \cite{KishimotoProc}.

It is worthwhile mentioning here that the method of normal form reductions 
has other uses besides proving unconditional uniqueness. 
For example, 
it has been used by Oh and Wang \cite{OhWang} to exhibit energy estimates in negative
Sobolev spaces  for the periodic 
fourth order NLS with cubic nonlinearity.  
Also, by combining the normal form reductions idea with $X^{s,b}$-analysis, 
Erdo\u{g}an and Tzirakis~\cite{ErdoganTzirakis} 
proved a 
nonlinear smoothing property for the periodic
Korteweg-de Vries equation,  
and more recently for DNLS on the real line by Erdo\u{g}an, Gurel, and Tzirakis~\cite{ErdoganTzirakis}. 

In the following we describe the normal form approach for DNLS on the
real line.

\subsection{The normal form method for DNLS}
As in the work of Takaoka \cite{Tak99}, we have to use a gauge
transformation\footnote{More recently, 
Pornnopparath \cite{Pornnopparath} showed the local well-posedness of \eqref{DNLS} 
in $H^{s}(\R)$, $s\ge \frac12$, \emph{without} 
using a gauge transformation. In fact, the same result is shown to hold 
for a more general nonlinearity than in \eqref{DNLS}, 
namely a generic polynomial in $(u, \overline{u}, \dx u, \dx \overline{u})$ where all monomials 
have degree $\ge 3$ and at most one derivative.}   
(i.e. a nonlinear change of variable $u\mapsto w$) in
order to remove the nonlinearity $2i|u|^2 \dx u$ from \eqref{DNLS}.
See also Remark~\ref{rmk:dxonconjug}.  This transformation changes
favorably the cubic nonlinearity but introduces a (pure-power)
quintic term. Therefore, we begin with the following \emph{gauged
  DNLS} (see Section \ref{sect2}):
\begin{equation}
  \label{dtweq}
  i \partial_t w+ \partial_x^2 w = -iw^2\partial_x \overline{w} - \frac12 |w|^4w \ ,\ t\in I\,,
\end{equation}
where $I$ (with $0\in I$) is a time interval on which a solution $u$
to \eqref{DNLS} exists.  By setting $v(t)=e^{-it\partial_x^2} w(t)$
(the interaction representation of $w$), one can rewrite the gauged
DNLS as
\begin{gather}
  \begin{split}
    \partial_t v= \Tcal (v)+ \Qcal (v)
    &\colonequals \Fcal^{-1}\bigg\{-i\int_{\xi=\xi_1-\xi_2+\xi_3}e^{i\Phi(\bxi)t}\xi_2\hat{v}(t,\xi_1)\overline{\hat{v}(t,\xi_2)}\hat{v}(t,\xi_3)d\xi_1d\xi_2\bigg\}\\
    \label{eq:interaction}
    &\quad \quad -\frac12 \big|e^{-it\partial_x^2}v(t)\big|^4
    e^{-it\partial_x^2}v(t),
  \end{split}
\end{gather}
where $\Fcal$ denotes the Fourier transform in the spatial variable, 
and the modulation function $\Phi(\bxi)$ is given by
\[\Phi(\bxi)=\Phi(\xi,\xi_1,\xi_2,\xi_3)\colonequals
  \xi^2-\xi_1^2+\xi_2^2-\xi_3^2.\] Thanks to the algebra property of
$H^s(\R)$, $s>\frac{1}{2}$, we may focus our attention to the cubic
nonlinearity $\Tcal(v)$.  Indeed, the quintic term $\Qcal(v)$ can be
estimated easily:
\[\|\mathcal{Q}(v)\|_{H^s(\R)} \lesssim \|v\|_{H^s(\R)}^5.\] Such an
estimate clearly does not hold for $\Tcal(v)$ due to the presence of
the spatial derivative (``the derivative loss issue'').  Hence, we
proceed to iteratively substitute this nonlinearity with (infinitely
many) terms which are easily controlled in the $H^s(\R)$-norm.

Let us take the spatial Fourier transform of (the Duhamel formulation
of) \eqref{eq:interaction} and we formally integrate by parts in the
temporal variable to obtain:
\begin{gather}
  \begin{split}
    \label{eq:step1nfr}
    \hat{v}(t,\xi)=\hat{v}(0,\xi) &- \int_{\xi=\xi_1-\xi_2+\xi_3}\frac{e^{i\Phi(\bxi)t}\xi_2}{\Phi(\bxi)}\hat{v}(t',\xi_1)\overline{\hat{v}(t',\xi_2)}\hat{v}(t',\xi_3)d\xi_1d\xi_2\bigg|_{t'=0}^t\\
    & + \int_0^t \int_{\xi=\xi_1-\xi_2+\xi_3} \frac{e^{i\Phi(\bxi)t}\xi_2}{\Phi(\bxi)} \partial_t \Big[\hat{v}(t',\xi_1)\overline{\hat{v}(t',\xi_2)}\hat{v}(t',\xi_3)\Big]d\xi_1d\xi_2dt'\\
    &+ \int_0^t \widehat{\Qcal(v)}(t',\xi) dt' \,.
  \end{split}
\end{gather}
We first note that we aim to overcome the derivative loss issue of
$\Tcal(v)$ by exploiting the denominator $\Phi(\bxi)$ after such an
integration by parts step, at least in an integration region where the
modulation function $\Phi(\bxi)$ is large (i.e. ``away from resonant''
contribution to $\Tcal(v)$).  On the other hand, when the modulation
function $\Phi(\bxi)$ is in a neighborhood of $0$ (i.e. ``almost
resonant'' contribution to $\Tcal(v)$), the denominator would actually
work against us, being impossible to handle the terms appearing in
\eqref{eq:step1nfr} directly in the $H^s(\R)$-norm.


In our analysis we distinguish two cases, namely (i) the almost
resonant case: $|\Phi(\bxi)|\leq N$ and (ii) the away from resonant
case: $|\Phi(\bxi)|>N$, for some suitably large threshold
$N=N(\|v_0\|_{H^s})$.  In the case (i), thanks to the restriction on
the modulation, we can directly estimate the contribution of
$\Tcal(v)$ from \eqref{eq:interaction} in $H^s(\R)$, $s>\frac{1}{2}$
(see Corollary 3.5).  In the integration region (ii), we proceed to
perform the integration by parts as in \eqref{eq:step1nfr}.

In view of \eqref{eq:interaction}, the second integral in
\eqref{eq:step1nfr} can be written as the sum of quintic and septic
terms.  Indeed, by assuming that the temporal derivative falls on the
first factor, the second integral in \eqref{eq:step1nfr} can be
essentially written as
\begin{align}
  \nonumber
  &  \int_0^t \int_{\xi=\xi_1-\xi_2+\xi_3} \frac{e^{i\Phi(\bxi)t}\xi_2}{\Phi(\bxi)} \Big(\widehat{\Tcal(v)}(t',\xi_1)+\widehat{\Qcal(v)}(t',\xi_1)\Big)\overline{\hat{v}(t',\xi_2)}\hat{v}(t',\xi_3)d\xi_1d\xi_2dt'\\
  \nonumber
  &\quad \ \sim\ \int_0^t \int_{\substack{\xi=\xi_1-\xi_2+\xi_3\\\xi_1=\xi_{11}-\xi_{12}+\xi_{13}}} \frac{e^{i(\Phi(\bxi)+\Phi(\bxi_1))t}\xi_2\xi_{12}}{\Phi(\bxi)}\hat{v}(\xi_{11})\overline{\hat{v}(\xi_{12})}\hat{v}(\xi_{13}) \overline{\hat{v}(\xi_2)}\hat{v}(\xi_3)d\xi_{11}d\xi_{12}d\xi_1d\xi_2dt'\\
  \label{eq:step2nfr}
  &\quad\qquad+\int_0^t \int_{\xi=\xi_1-\xi_2+\xi_3} \frac{e^{i\Phi(\bxi)t}\xi_2}{\Phi(\bxi)} \widehat{\Qcal(v)}(t',\xi_1)\overline{\hat{v}(t',\xi_2)}\hat{v}(t',\xi_3)d\xi_1d\xi_2dt' \,,
\end{align}
where
$\Phi(\bxi_1)\colonequals\Phi(\xi_1,\xi_{11},\xi_{12},\xi_{13})$.
Although we have an $H^s(\R)$-estimate for the last term in
\eqref{eq:step2nfr}, the contribution due to $\Tcal(v)$ (i.e. the
quintic term in \eqref{eq:step2nfr}) suffers from the same derivative
loss issue as $\Tcal(v)$ itself.  The idea now is to repeat the
previous two-steps iteration.  First, we split the domain of the
second integral in \eqref{eq:step2nfr} again into (i) the almost
resonant case: $|\Phi(\bxi)+\Phi(\bxi_1)|\leq N_1$ where we can
establish an $H^s(\R)$-estimate and (ii) the away from resonant case:
$|\Phi(\bxi)+\Phi(\bxi_1)|>N_1$.  We then integrate by parts
\emph{only} in (ii) and exploit the gain of the denominator
$\Phi(\bxi)+\Phi(\bar{\xi_1})$ (the price paid being additional
nonlinearities of higher degrees).  It turns out that it is helpful to
chose the threshold $N_1\sim |\Phi(\bxi)|$ and we point out that at
this stage we have as well $|\Phi(\bxi)|>N$.  Regarding the two left
out terms, namely when the time derivative falls on the $k$th factor
($k=2,3$), we mention here that the factor
$e^{i(\Phi(\bxi)+\Phi(\bxi_1))t} \xi_1 \xi_{12}$ above changes to
$e^{i(\Phi(\bxi)+\Phi(\bxi_k))t} \xi_1 \xi_{k2}$ 
and that we use the same strategy as described above.

After $J$ iterations we derive the following equation
\begin{gather}
  \begin{split}
    \label{DuhFcompl}
    v(t) = &\, v_0 + \int_0^t \Qcal(v)(t')dt' + \sum_{j=2}^{J+1}
    \bigg(\Tcal_0^{(j)}(v)(t) - 
    \Tcal_0^{(j)}(v)(0) \bigg)
    + \sum_{j=2}^{J+1} \int_0^t \Tcal^{(j)}_{\Qcal}(v)(t') dt' \\
    &\,\,+ \sum_{j=1}^{J} \int_0^t \Tcal^{(j)}_{\Tcal,1}(v)(t') dt' +
    \int_0^t \Tcal^{(J+1)}_{\Tcal}(v)(t') dt'
  \end{split}
\end{gather}
and the nonlinearity $\Tcal^{(J+1)}_{\Tcal}(v)$ is passed on to the
next iteration.  In comparing \eqref{eq:interaction} with
\eqref{DuhFcompl}, notice that we have replaced the nonlinearity
$\Tcal(v)$ by several terms whose origin (at iteration $j$) we briefly
explain here: the $\Tcal_0^{(j)}(v)$ term denotes the boundary terms
that appear when integrating by parts, $\mathcal{T}_\Qcal^{(j)}(v)$
stands for the terms corresponding to replacing $\dt v$ by $\Qcal(v)$,
$\Tcal_{\Tcal,1}^{(j)}$ stands for the terms corresponding to replacing
$\dt v$ by $\Tcal(v)$ followed by restricting the appropriate
modulation function to the almost resonant case, and finally
$\Tcal^{(J+1)}_{\Tcal}(v)$ is ``the remainder term'' which is passed
to the $(J+1)$th iteration.
Since $\dt $ may fall on any of the factors of $v$, it becomes
apparent that one has to manage the bookkeeping of terms (whose number
grows facorially in $J$).  We accomplish this by using the notion of
\emph{ordered trees} as in the work of the second author together with
Kwon and Oh \cite{KOY}.  See also the paper by Christ
\cite{ChristPowerSeries} in which a precursor notion was used.

The key point to be made at this stage is that we manage to show that,
for \emph{fixed} $N$,
\begin{equation}
  \label{rmdtozero}
  \|\Tcal_{\Tcal}^{(J+1)}(v)\|_{H^{s-1}(\R)} \to 0 \,, 
\end{equation}
as $J\nearrow \infty$.  While we do not have control of the remainder
term in the $H^{s}(\R)$-norm, the remainder term vanishes in the limit
in a \emph{weaker} topology than the strong $H^s(\R)$-topology (see
Subsection~\ref{subsect:4p1}).  This fact together with
$H^{s-1}(\R)$-estimates similar to \eqref{rmdtozero} (see
Section~\ref{sect:justif}) allow us to prove that any solution
$v\in C(I;H^s(\R))$ to \eqref{eq:interaction} necessarily satisfies
(in $H^s(\R)$) the normal form equation:
\begin{gather}
  \begin{split}
    v(t) = &\, v_0 + \int_0^t \Qcal(v)(t')dt' + \sum_{j=2}^{\infty}
    \bigg(\Tcal_0^{(j)}(v)(t) - 
    \Tcal_0^{(j)}(v)(0) \bigg)
    + \sum_{j=2}^{\infty} \int_0^t \Tcal^{(j)}_{\Qcal}(v)(t') dt' \\
    &\,\,+ \sum_{j=1}^{\infty} \int_0^t \Tcal^{(j)}_{\Tcal,1}(v)(t')
    dt' \,,
    \label{DuhFinfty}
  \end{split}
\end{gather}
for all $t\in I$.

The analysis for the equation \eqref{DuhFinfty} is simple: we apply a
fixed point argument directly in the $C(I;H^s(\R))$-norm, without
relying on extra harmonic analytic tools.  Indeed, we can write all
the nonlinear terms in \eqref{DuhFinfty} as iterated applications of a
single trilinear form \eqref{locmodNcal}.  
Once we have the $H^s(\R)$-estimate for this
simple trilinear form (Lemma~\ref{locmodest}), we obtain control of
all the terms in \eqref{DuhFinfty} (see
Section~\ref{Sect:strongests}).  This is a very efficient method to
deal with the infinite series of nonlinearities and it was applied
before in \cite{KishimotoProc,KishimotoDNLS,KOY}.  Showing
\eqref{rmdtozero} also relies on this idea; however, for this purpose
one needs two ``building blocks'', namely the $H^{s-1}(\R)$-estimates
of $\dt v$ for $v\in C(I; H^s(\R))$ solution to \eqref{eq:interaction}
and of a second trilinear form (in an ``away from resonant''
integration region).  See Corollary~\ref{lem:dtvest} and
Lemma~\ref{locmodNcalw}.
For an exposition of this idea we refer the reader to the report paper
by Kishimoto \cite{KishimotoProc} (in particular, see the meta-theorem
\cite[Theorem~1]{KishimotoProc}).

In summary, the method applied in this work is antipodal to that of
the Fourier restriction norm method (as applied by Takaoka
\cite{Tak99} for DNLS): we first derive a \emph{complicated Duhamel
  formula}, that is the normal form equation \eqref{DuhFinfty}, after
which the analytical part is simple.  In contrast, one needs a more
involved analysis when using the $X^{s,b}$-norms (i.e. the Fourier
restriction norm method) given by
\begin{equation}
  \label{defn:Xsb}
  \|w\|_{X^{s,b}}= \big\| \jb{\dx}^s \jb{\dt}^b e^{-it\partial_x^2} w(t) \big\|_{L_{t,x}^2(\R^2)} 
  \quad (s,b\in \R)
\end{equation}
on the simple Duhamel formula of \eqref{dtweq}.  For a similarity,
notice that the interaction representation of $w(t)$ also plays a role
in the Fourier restriction norm method.  In the ``denominator games''
specific to the Fourier restriction norm method, one essentially
overcomes the derivative loss issue with a denominator
$|\Phi(\bxi)|^b$ with $b\approx \frac12$.  In the method employed
here, due to the integration by parts (see e.g. \eqref{eq:step1nfr}),
we benefit from a full power $|\Phi(\bxi)|$.

Finally, we emphasize that the proviso for the scheme of infinite
iterations of normal form reductions to work is showing that the
remainder term vanishes in the limit.  In some sense, this represents
the heavier analytical part of this method, namely identifying some
weaker norm than the $C(I; H^s(\R))$-norm in which one can get
\eqref{rmdtozero}.

\subsection{Comments and remarks}

For DNLS on the real line, Yin Yin Su Win \cite{Win08} established its
unconditional well-posedness in the energy space, i.e., for $s=1$.
Indeed, by modifying the $X^{s,b}$-multilinear estimates in
\cite{Tak99}, the author of \cite{Win08} showed the uniqueness of
solutions to DNLS in $X^{\frac{1}{2},\frac{1}{2}}_T$ (here,
$X^{s,b}_T$ simply denotes a local in time version of $X^{s,b}$
defined via \eqref{defn:Xsb}).  Now, uniqueness of solutions in
$X^{\frac{1}{2},\frac{1}{2}}_T$ implies unconditional uniqueness of
solutions to DNLS in $H^1(\R)$.  Indeed, this follows from arguing by
interpolation (of $X^{s,b}$-spaces): first, if
$u\in C([-T,T];H^1(\R))$, then clearly
$u\in X_T^{1,0} = L^2([-T,T];H^1(\R))$; second, by the algebra
property of $C([-T,T];H^1(\R))$ we have
$\dx(|u|^2 u) \in C([-T,T]; L^2(\R)) \subset L^2([-T,T]; L^2(\R))$ and
thus $u=(i\dt + \dx^2)^{-1} \big(i\dx(|u|^2u)\big) \in X^{0,1}_T$;
third, by interpolation, any solution $u\in C([-T,T];H^1(\R))$ to
\eqref{DNLS} is contained in $X_T^{\frac{1}{2},\frac{1}{2}}$ and thus
it must be unique.  This strategy does not work for $s<1$ because the
key trilinear estimate is known to fail in $X^{s,b}$ with $s<\frac12$,
for any $b\in\R$ (see \cite[Proposition~3.3]{Tak99}).

For DNLS on the torus, Kishimoto \cite{KishimotoDNLS} proved its
unconditional well-posedness in $H^s(\T)$, for $s>\frac12$.  In
addition to \cite{KOY}, our implementation of the infinite iteration
of normal form reductions to prove Theorem~\ref{mainthm} follow ideas
presented in \cite{KishimotoProc,KishimotoDNLS}, specifically in
making use of the trilinear forms $\Tcal_{\Phi}$ and
$\Tcal^{\textup{w}}_{|\Phi|>M}$ in Sections~\ref{Sect:strongests} and
\ref{sect:weakests}.  In contrast, in \cite{KOY} (handling the cubic
NLS and mKdV equations on the real line in Sobolev spaces) and in
\cite{ForlanoOh} (handling the cubic NLS in almost critical spaces),
the approach is to prove ``strong and weak localized modulation
estimates'' (SLME and WLME)
and then use more intricate thresholds to separate the almost resonant
and away from resonant integration regions at each iteration.
Although we can still prove a useful SLME for DNLS in order to
establish the $H^s(\R)$-estimates for all nonlinearities in a normal form equation derived from DNLS, 
there seems to be no useful corresponding WLME. 

Finally, we include here a corollary to Theorem~\ref{mainthm}
regarding the global well-posedness of DNLS.  We recall that
Colliander, Keel, Staffilani, Takaoka, and Tao
\cite{CKSTT,CKSTTrefined} introduced the $I$-method and showed that it
is in fact globally well-posed, provided that $s>\frac12$ and
$\|u_0\|_{L^2}^2<2\pi$.  Miao, Wu, and Xu \cite{MiaoWuXu} reached the
end-point regularity $s=\frac12$, under the same condition on the
$L^2$-norm of the initial data.  The $L^2$-norm threshold on initial
data was improved\footnote{Prior to \cite{GuoWu}, in \cite{Wu2013,WuAPDE2},  Wu first
obtained $L^2$-norm threshold improvements for energy-space initial
data.} to $\|u_0\|_{L^2}^2<4\pi$ by Guo and Wu \cite{GuoWu}
who showed global well-posedness of DNLS in $H^s(\R)$,
$s\geq \frac12$. 

Taking into account Theorem~\ref{mainthm} and the result of
\cite{GuoWu}, we obtain the following:

\begin{corollary}
  \label{cor}
  Let $s>\frac12$, $u_0\in H^s(\R)$ with $\|u_0\|^2_{L^2}<4\pi$.
  Then, DNLS is unconditionally globally well-posed in $H^s(\R)$.
\end{corollary}

Although we do not pursue the question of global well-posedness of DNLS in this paper, 
we would like to point out that above the mass threshold $4\pi$, the
question of whether all solutions to \eqref{DNLS} extend globally in time is
not settled for low-regularity initial data. 
We mention here two recent papers that are relevant to this question. 
First, for $H^1(\R)$-initial data,
by using variational analysis of soliton solutions, Fukaya, Hayashi,
and Inui \cite{FHI} gave a sufficient condition for the global
well-posedness of \eqref{DNLS} covering the result of Wu \cite{WuAPDE2}. 
Second, by using the inverse
 scattering method, Jenkins, Liu, Perry, and Sulem \cite{JLPS} (see also references therein) 
 proved 
 that all solutions started with initial data in the weighted Sobolev
 space $H^{2,2}(\R)$ with the norm
$\|u\|_{H^{2,2}(\R)}=\Big(\|\jb{\,\cdot\,}^2u(\,\cdot\,)\|_{L^2(\R)}^2+\|u''\|_{L^2(\R)}^2\Big)^{1/2}$
 exist for all times.

\subsection{Organization of the paper} 
In Section~\ref{sect2}, we perform normal form reductions and
transform the (gauged) DNLS equation into an equation which is more
complicated algebraically, but simpler analytically.  The proofs of
the crucial estimates are given in Sections 3 and 4.  In
Section~\ref{sect:justif}, we rigorously justify the various
operations from Section~\ref{sect2} for rough solutions to DNLS.
Finally, in Section~\ref{sect:pfmainthm} we put the pieces together
and give the proof of Theorem~\ref{mainthm}.

\subsection{Notation}
We use $A\lesssim B$ to denote the estimate that $A\leq CB$ for some
constant $C$ which may vary from line to line and depend on various
parameters.  We use $A\sim B$ to denote the statement that
$A\lesssim B\lesssim A$. We also use $A\ll B$ if $A\leq \epsilon B$,
where $\epsilon$ is a small absolute constant.  For an integrable
function $f(x)$ with $x\in \R$, we use the Fourier transform
convention
\[\Fcal(f)(\xi)=\widehat{f}(\xi)\colonequals \int_\R f(x)e^{-i x \xi}
  \,dx.\] \noi We denote $S(t)=e^{it\partial_x^2}$ the linear
propagator for the linear Schr\"odinger equation
$\partial_t u=i\partial_x^2 u$.

We include in Appendix~\ref{appdxA} the notion of ordered trees and
related terminology as introduced in \cite{KOY} in order to make our
paper self-contained.

\section{The normal form equation}
\label{sect2}

In this section, we \emph{formally} derive a {normal form equation} for a so-called gauged DNLS equation.  
First, we use a gauge transformation to remove the nonlinear term $2i |u|^2\partial_x u$ 
from the right-hand side of \eqref{DNLS} 
at the expense of introducing a (pure power) quintic nonlinear term 
-- see \eqref{gDNLS} below. 
Then, we apply an infinite iteration of normal form reductions 
to transform the gauged DNLS into a new equation 
involving infinite series of nonlinearities of arbitrarily high degrees. 
To this end, we employ the machinery developed in \cite{KOY}.

We use  the following gauge transformation 
\begin{equation}
\label{gaugetr}
u(t,x)\mapsto w(t,x):=\exp\left(-i\int_{-\infty}^x |u(t,y)|^2dy\right) u(t,x) .
\end{equation}
Notice that this is an autonomous transformation, i.e. it does not depend explicitly on the time variable. 
Thus, equation \eqref{DNLS} is transformed into 
the 
\emph{gauged DNLS}: 
\begin{equation}
\label{gDNLS}
i \partial_t w+ \partial_x^2 w = -iw^2\partial_x \overline{w} - \frac12 |w|^4w .
\end{equation}
This nonlinear transformation \eqref{gaugetr} goes back 
to the works of Hayashi \cite{Hayashi93} and 
Hayashi and Ozawa \cite{HaOz92}. See also \cite{Lee89}. 
It is well known by now (see \cite{Tak99}) 
that the cubic nonlinearity with the derivative falling on  the complex-conjugate  factor can be handled using the 
Fourier restriction norm method, 
whereas the cubic term $|u|^2\partial_x u$ fails to have a useful estimate. 
It turns out that this is also the case when employing the normal form approach, 
namely we have to  remove  the bad nonlinearity before renormalizing the equation 
-- see also Remark~\ref{rmk:dxonconjug}. 
We can transfer a well-posedness result on the gauged DNLS equation back to the original DNLS equation 
with the following: 

\begin{lemma}[{\cite{CKSTTrefined}}]
\label{gaugetrLip}
Let $s\ge 0$. 
The mapping $u\mapsto w$ defined by \eqref{gaugetr} 
is bi-Lipschitz on $H^s(\R)$.  
\end{lemma}

Next, we denote $S(t):=e^{it\partial_x^2}$ 
and we use the change of variable 
$v(t)=S(-t)w(t)$ (the interaction representation variable). 
Then, the equation \eqref{gDNLS} becomes 
\begin{equation}
\label{irvgDNLS}
\partial_t v = \mathcal{Q}(v) + \mathcal{T}(v),
\end{equation}
where we denoted the quintic and the cubic nonlinear terms respectively by: 
\begin{align} 
\label{Qcal}
\mathcal{Q}(v) &:=-\frac12 \big|S(t)v(t)\big|^4 S(t)v(t),\\
\label{Tcal}
\mathcal{T}(v) &:=-i \big(S(t)v(t)\big)^2 \partial_x \overline{S(t)v(t)}.
\end{align}
In what follows we exploit the oscillatory nature of the Fourier transform of $\mathcal{T}(v)$.

With a slight abuse of notation\footnote{Note that when all the entries of the trilinear operator are the same, we write $\mathcal{T}(v)$ instead of $\mathcal{T}(v,v,v)$.}, 
let us introduce the trilinear operator $\mathcal{T}$ defined by 
\begin{align}
\label{defn:Tcal}
\mathcal{F}\Big[\mathcal{T}(v_1,v_2,v_3)\Big](t,\xi) 
 &=  \int_{\xi=\xi_1-\xi_2+\xi_3}  e^{i\Phi(\bxi) t} \xi_2 
\widehat{v_1}(\xi_1) \overline{\widehat{v_2}(\xi_2)} \widehat{v_3}(\xi_3) d\xi_1d\xi_2 \,,
\end{align}
where the phase is given by
\begin{equation}
\label{defn:Phi}
\Phi(\bar{\xi}) := \xi^2-\xi_1^2 +\xi_2^2-\xi_3^2 .
\end{equation}
Notice that on the convolution hyperplane $\xi=\xi_1-\xi_2+\xi_3$, we have 
\begin{equation*}
\Phi(\bar{\xi}) = 2(\xi-\xi_1)(\xi-\xi_3) = 2(\xi_2-\xi_1)(\xi_2-\xi_3) . 
\end{equation*}
Since it is determined by the linear part of the equation, 
the function $\Phi(\bar{\xi})$ is the same as 
the \emph{modulation function} for the cubic NLS equation in \cite{KOY}, 
but the trilinear operator is different due to the presence of the derivative in the cubic nonlinearity.

Since for $s>\frac12$, $H^{s}(\R)$ is a Banach algebra, 
the quintic term can be estimated easily: 
\begin{equation}
\label{NfrakHsest}
\|\mathcal{Q}(v)\|_{H^s(\R)} \lesssim  \|v\|_{H^s(\R)}^5.
\end{equation}
Due to the derivative loss in the cubic term, 
$\mathcal{T}$ does not have a similar estimate in $H^s(\R)$, even though $s>\frac12$.  
Therefore we proceed to renormalize this nonlinearity 
by means of normal form reductions (NFR). 

\begin{remark}
Throughout this paper, when the complex conjugate sign on $v(\xi)$ does not play any significant role in the analysis, we drop the complex conjugate sign.
Also, we often drop the complex number $i$ and use $1$ for $\pm 1$ and $\pm i$.
\end{remark}

\subsection{The first step of NFR} 
The idea is to exploit the oscillatory factor of the convolution integral in \eqref{defn:Tcal}, 
and so we apply integration by parts on a domain of integration where $|\Phi(\bar{\xi})|>N$, 
for some threshold $N>1$ to be chosen later.  
We first decompose 
\begin{equation}
\label{firstdecomp}
\mathcal{T}(v) = \mathcal{T}_1(v) + \mathcal{T}_2(v),
\end{equation}
where $\mathcal{T}_2(v)$ is defined as $\mathcal{T}(v)$ (see \eqref{defn:Tcal} above),  
but the integration is further restricted to the domain 
\begin{equation*}
C_0 =C_0(\xi):= \Big\{(\xi_1,\xi_2,\xi_3)\in \R^3 : \xi=\xi_1-\xi_2+\xi_3,\ |\Phi(\bxi)| > N\Big\}
\end{equation*}
embedded in the convolution hyperplane $\xi=\xi_1-\xi_2+\xi_3$
and let $\mathcal{T}_1(v):= \mathcal{T}(v) - \mathcal{T}_2(v)$. 
Thanks to the modulation restriction, the term $ \mathcal{T}_1(v)$ enjoys 
a sufficiently good $H^s(\R)$-estimate -- see Lemma~\ref{locmodest} below. 
For the remainder term $\mathcal{T}_2(v)$, 
we apply differentiation by parts\footnote{Here, ``differentiation by parts'' 
means usual integration by parts (with respect to the time variable) 
in the Duhamel formulation of \eqref{irvgDNLS}, without writing explicitly the time integration. 
In other words,
$$ \Tcal_2(v)(t,\xi) = \partial_t \Big[ \Tcal_0^{(2)}(v)(t',\xi)\Big]  +  \Tcal^{(2)}(v)(t,\xi)$$
stands for 
$$ \int_0^t
\Tcal_2(v)(t',\xi)dt'  = \Big[ \Tcal_0^{(2)}(v)(t',\xi)\Big]_{t'=0}^{t'=t}  +  \int_0^t \Tcal^{(2)}(v)(t',\xi) dt' .
$$
} 
in order to renormalize it. 
To ease the writing,  
we drop the complex conjugate, the Fourier transform notation, 
and the complex constants of modulus one in front of the nonlinearities. 
We have: 
\begin{align}
\nonumber
\Tcal_2(v)(t,\xi) &=  \partial_t \Bigg[
\int_{\substack{\xi=\xi_1-\xi_2+\xi_3\\ |\Phi(\bar{\xi})|>N}} 
 \frac{e^{i\Phi(\bar{\xi}) t } \xi_2}{\Phi(\bar{\xi})}   v(t,\xi_1)v(t,\xi_2)v(t,\xi_3) d\xi_1d\xi_2  \Bigg]\\
\nonumber
 &\qquad\qquad - 
 \int_{\substack{\xi=\xi_1-\xi_2+\xi_3\\ |\Phi(\bar{\xi})|>N}} 
\frac{e^{i\Phi(\bar{\xi}) t}\xi_2 }{\Phi(\bar{\xi})} 
  \partial_t \big(v(t,\xi_1)v(t,\xi_2)v(t,\xi_3) \big) d\xi_1d\xi_2 \\
\nonumber
  &=: \partial_t \Big[ \Tcal_0^{(2)}(v)(t,\xi)\Big]  +  \Tcal^{(2)}(v)(t,\xi).
\end{align}

\noi
Let us start employing the ordered tree notation from Appendix~\ref{appdxA}. 
At this stage, we can express everything in terms of $\TT_1$, the sole ternary tree of the first generation. 
With $\mu_1:= \Phi(\bar\xi)$, 
the nonlinearities $\Tcal_0^{(2)}(v)$,  $\Tcal^{(2)}(v)$ can be written as follows: 
\begin{align}
\label{Tcal02}
\Tcal_0^{(2)}(v)(t,\xi) &= \int_{\bmxi\in\Xi_{\xi}(\TT_1)}
 \ind_{C_0}\frac{e^{i\mu_1 t}\xi^{(1)}_2}{\mu_1}  \prod_{a\in\TT_1^{\infty}} v(t,\xi_a)\,,\\
\label{Tcal2}
\Tcal^{(2)}(v)(t, \xi) &= 
\int_{\bmxi\in\Xi_{\xi}(\TT_1)}
 \ind_{C_0}\frac{e^{i \mu_1 t}\xi^{(1)}_2}{\mu_1} 
  \partial_t \bigg( \prod_{a\in\TT_1^{\infty}} v(t,\xi_a) \bigg)\,.
 \end{align}
By using the product rule and supposing $v$ is a smooth solution of \eqref{irvgDNLS}, 
we get 
\begin{equation*}
\Tcal^{(2)}(v)= \Tcal^{(2)}_{\Qcal}(v) + \Tcal_{\Tcal}^{(2)}(v)\,. 
\end{equation*}
On the right side above, 
$\Tcal^{(2)}_{\Qcal}(v)$ is the sum of three septic terms, 
corresponding to replacing $\partial_t v(t,\xi_b)$ by $\Qcal(v)(t,\xi_b)$, $b\in \TT_1^{\infty}$. 
Similarly, 
$\Tcal^{(2)}_{\Tcal}(v)$  is the sum of three quintic terms, 
corresponding to replacing $\partial_t v(t,\xi_b)$ by $\Tcal(v)(t,\xi_b)$, $b\in \TT_1^{\infty}$. 
More precisely, we have 

 \begin{align}
 \label{TcalQ2}
\Tcal_{\Qcal}^{(2)}(v)(\xi) &:=  \sum_{b\in\TT_1^{\infty}}
 \int_{\bmxi\in\Xi_{\xi}(\TT_1)}
 \ind_{C_0^c}\frac{e^{it\mu_1}\xi^{(1)}_2 }{\mu_1} 
   \Qcal(v)(\xi_b) \prod_{a\in\TT_1^{\infty}\setminus\{b\}} v(\xi_a)\\
\label{TcalT2}
\Tcal_{\Tcal}^{(2)}(v)(\xi) &:= \sum_{b\in\TT_1^{\infty}}
 \int_{\bmxi\in\Xi_{\xi}(\TT_1)}
 \ind_{C_0^c}\frac{e^{it\mu_1}\xi^{(1)}_2 }{\mu_1} 
   \Tcal(v)(\xi_b) \prod_{a\in\TT_1^{\infty}\setminus\{b\}} v(\xi_a)
\end{align}

Thus, if $v$ is a smooth solution of \eqref{irvgDNLS}, then it is also a solution of 
\begin{equation}
\label{NF1gDNLS}
\partial_t v = \Qcal(v) + \partial_t \Tcal_0^{(2)}(v) + \Tcal^{(1)}_{\Tcal,1}(v) 
+ \Tcal_{\Qcal}^{(2)}(v) +  \Tcal_{\Tcal}^{(2)}(v),
\end{equation}
where we set $\Tcal^{(1)}_{\Tcal,1}(v) := \Tcal_1(v)$ 
for the sake of consistency with subsequent NFR steps. 
It turns out that we can establish sufficiently good estimates for all of the nonlinear terms of \eqref{NF1gDNLS}, 
except for those in $\Tcal_{\Tcal}^{(2)}(v)$. Therefore, we proceed to renormalize them.

\subsection{The second step of NFR}

For the sake of clarity, 
let us write $\Tcal_{\Tcal}^{(2)}(v)$ defined in \eqref{TcalT2} first without appealing to the terminology of Appendix~\ref{appdxA}, 
and then in the compact writing facilitated by the ordered trees notation: 
\begin{align}
\nonumber
\Tcal_{\Tcal}^{(2)}(v)(\xi) = &
\int_{\substack{\xi=\xi_1-\xi_2+\xi_3\\ \xi_1=\xi_{11}-\xi_{12}+\xi_{13}}} 
\ind_{C_0^c} \frac{e^{i\Phi(\bar\xi)t} \xi_2}{\Phi(\bar\xi)} \Big(e^{i\Phi(\overline{\xi_1})t}\xi_{12} \Big)
 v(\xi_{11}) v(\xi_{12}) v(\xi_{13}) v(\xi_2) v(\xi_3)\\
 \nonumber
 & + 
 \int_{\substack{\xi=\xi_1-\xi_2+\xi_3\\ \xi_2=\xi_{21}-\xi_{22}+\xi_{23}}} 
\ind_{C_0^c} \frac{e^{i\Phi(\bar\xi)t} \xi_2}{\Phi(\bar\xi)} \Big(e^{i\Phi(\overline{\xi_2})t}\xi_{22} \Big)
 v(\xi_1)v(\xi_{21}) v(\xi_{22}) v(\xi_{23}) v(\xi_3)\\
 \nonumber
  & + 
 \int_{\substack{\xi=\xi_1-\xi_2+\xi_3\\ \xi_3=\xi_{31}-\xi_{32}+\xi_{33}}} 
\ind_{C_0^c} \frac{e^{i\Phi(\bar\xi)t} \xi_2}{\Phi(\bar\xi)} \Big(e^{i\Phi(\overline{\xi_3})t}\xi_{32} \Big)
 v(\xi_1)v(\xi_2) v(\xi_{31}) v(\xi_{32}) v(\xi_{33})\\
 \label{TcalT2cpct}
 =& 
\sum_{\TT \in\Tfrak(2)} \int_{\bmxi\in\Xi_{\xi}(\TT)}
 \ind_{C_0^c}\frac{e^{i\mu_1 t} \xi^{(1)}_2}{\mu_1} 
 \Big( e^{i\mu_2t} \xi^{(2)}_2 \Big)\prod_{a\in\TT^{\infty}} v(\xi_a),
\end{align}
where $\Phi(\overline{\xi_j})=\Phi(\xi_j,\xi_{j1},\xi_{j2},\xi_{j3})$ for $1\leq j\leq 3$.
 Notice that, in \eqref{TcalT2cpct}, the phase is $\mu_1+\mu_2$, where  
  $\mu_1$ is the same as in the first step of NFR, i.e. $\mu_1=\Phi(\bxi)$, and 
$$\mu_2:=
\Phi(\overline{\xi^{(2)}})=2( \xi_2^{(2)}-\xi_1^{(2)})(\xi_2^{(2)}- \xi_3^{(2)})\,,$$
for $\bmxi\in\Xi_{\xi}(\TT)$. 
We now decompose 
\begin{equation*}
\Tcal_{\Tcal}^{(2)}(v) = \Tcal_{\Tcal,1}^{(2)}(v) + \Tcal_{\Tcal,2}^{(2)}(v)\,,
\end{equation*}
i.e. each term of the sum in  \eqref{TcalT2cpct} 
is split into two parts corresponding to further restricting the domain of integration to 
\begin{equation*}
C_1=C_1(\xi; \TT):= \big\{ \bmxi\in \Xi_{\xi}(\TT) : |\mu_1+\mu_2|\le \beta_1 |\mu_1|  \big\}
\end{equation*}
and its complement, respectively, where $\beta_1\ge 2$ is to be chosen later.  
By Lemma~\ref{lem:TcalTcal1Jp1} below, 
we have $H^s(\R)$-estimates for the terms in $\Tcal_{\Tcal,1}^{(2)}(v)$. 
For the remainder $\Tcal_{\Tcal,2}^{(2)}(v)$,
we apply differentiation by parts for all of its three terms. 
Thus 
by working with the ordered trees notation, 
we have\footnote{Given 
an ordered tree $\TT_2$ with $\TT_1$ denoting its first generation tree, 
for $A_1\subseteq \Xi(\TT_1)$, $A_2\subset \Xi(\TT_2)$, we define by a slight abuse of notation, 
$A_1\cap A_2:=\{\bmxi\in A_2 : \bmxi|_{\TT_1}\in A_1\}$. 
Inductively, 
this definition is generalized to higher generation ordered trees as follows: 
if $\TT_{J+1}$ is an ordered tree with chronicle $\{\TT_j\}_{j=1}^{J+1}$ 
and $A_j\subseteq \Xi(\TT_j)$, $j=1,2,\ldots, J+1$, 
then $A_1\cap A_2\cap\ldots\cap A_{J+1} := 
\{\xi\in \Xi(\TT_{J+1}) : \xi|_{\TT_{J}}\in A_1\cap A_2\cap\ldots\cap A_J \}$.}
\begin{align}
\nonumber
\Tcal_{\Tcal,2}^{(2)}(v)(t,\xi) &=  \partial_t \Bigg[
\sum_{\TT\in\Tfrak(2)} \int_{\bmxi\in\Xi_{\xi}(\TT)}
 \ind_{C_0\cap C_1^c}\frac{e^{i(\mu_1+\mu_2)t}}{\mu_1(\mu_1+\mu_2)} 
 \xi^{(1)}_2\xi^{(2)}_2 \prod_{a\in\TT^{\infty}} v(\xi_a) 
 \Bigg]\\
\nonumber
 &\qquad\qquad - 
\sum_{\TT\in\Tfrak(2)} \int_{\bmxi\in\Xi_{\xi}(\TT)}
 \ind_{C_0\cap C_1^c}\frac{e^{i(\mu_1+\mu_2)t}}{\mu_1(\mu_1+\mu_2)} 
 \xi^{(1)}_2\xi^{(2)}_2 \partial_t\bigg(\prod_{a\in\TT^{\infty}} v(\xi_a) \bigg)
 \\
\nonumber
  &=: \partial_t \Big[ \Tcal_0^{(3)}(v)(t,\xi)\Big]  +  \Tcal^{(3)}(v)(t,\xi).
\end{align}
By using the product rule and the assumption that $v$ is a smooth solution of \eqref{irvgDNLS}, 
we get 
\begin{equation*}
\Tcal^{(3)}(v)= \Tcal^{(3)}_{\Qcal}(v) + \Tcal_{\Tcal}^{(3)}(v)\,,
\end{equation*}
and the equation for $v$ becomes
\begin{equation*}
\partial_t v = \Qcal(v) 
+ \sum_{j=2}^3 \Tcal_0^{(j)}(v) 
+ \sum_{j=1}^2\Tcal^{(j)}_{\Tcal,1}(v) 
+ \sum_{j=2}^3\Tcal_{\Qcal}^{(j)}(v) 
+  \Tcal_{\Tcal}^{(3)}(v).
\end{equation*}

The last term $\Tcal_{\Tcal}^{(3)}(v)$ is passed to the next step in the iterative procedure. 
As we believe the iterative procedure became clear, let us present the general step of normal form reductions.

\subsection{The $J$th step of NFR}
We now write down the terms that appear in the $J$th step of normal form reductions. 
We decompose 
$\Tcal_{\Tcal}^{(J)}(v) = \Tcal_{\Tcal,1}^{(J)}(v) + \Tcal_{\Tcal,2}^{(J)}(v)\,,$
corresponding to further restricting the domain of integration of $\Tcal_{\Tcal}^{(J)}(v)$ to
\begin{equation*}
C_{J-1}=C_{J-1}(\xi; \TT) :=\Big\{ \bmxi \in\Xi_{\xi}(\TT) :  
\big|\widetilde{\mu}_{J-1} +\mu_{J} \big| \le \beta_{J-1} |\wt\mu_{J-1}|  \Big\}
\end{equation*}
and its complement, respectively, where $\beta_{J-1}\ge 2$ is to be chosen later (See \ref{rmk:choicebetas}). 
After differentiation by parts and by using the equation \eqref{irvgDNLS}, we are led to 
\begin{align}
\label{NFRJdiffparts}
 \Tcal_{\Tcal,2}^{(J)}(v)(t,\xi) &= \partial_t \Big[ \Tcal_0^{(J+1)}(v)(t,\xi)\Big]  
 +  \Tcal_{\Qcal}^{(J+1)}(v)(t,\xi) +\Tcal_{\Tcal}^{(J+1)}(v)(t,\xi)\,,
\end{align}
where 
the terms on the right-hand side are given by the following  formulae: 
\begin{align}
\label{defnTcal0Jp1}
\Tcal_0^{(J+1)}(v)(\xi)&= \sum_{\TT\in\TF(J)} 
 \int_{\bmxi\in\Xi_{\xi}(\TT)} 
  \ind_{F_J} 
  \bigg(\prod_{j=1}^{J} \frac{e^{i {\mu}_{j}t } \xi^{(j)}_2}{\widetilde{\mu}_j} \bigg) \bigg(\prod_{a\in\TT^{\infty}} v(\xi_a)\bigg)\\
\label{defnTcalQcalJp1}
\Tcal^{(J+1)}_{\Qcal}(v)(\xi)  &= \sum_{\TT\in\TF(J)} \sum_{b\in\TT^{\infty}} 
\int_{\bmxi\in\Xi_{\xi}(\TT)} 
 \ind_{F_J}  \bigg(\prod_{j=1}^{J} \frac{e^{i {\mu}_{j}t } \xi^{(j)}_2}{\widetilde{\mu}_j} \bigg) \bigg(\Qcal(v)(\xi_b) \prod_{\substack{a\in\TT^{\infty}\\ a\neq b}} v(\xi_a)\bigg)\\
\label{defnTcalTcalJp1}
\Tcal^{(J+1)}_{\Tcal}(v)(\xi)&= \sum_{\TT\in\TF(J+1)} 
\int_{\bmxi\in\Xi_{\xi}(\TT)} 
  \ind_{F_J}   \bigg(\prod_{j=1}^{J} \frac{e^{i {\mu}_{j}t } \xi^{(j)}_2}{\widetilde{\mu}_j} \bigg) 
   \big(e^{i{\mu}_{J+1} t} \xi^{(J+1)}_2 \big)
   \bigg(\prod_{a\in\TT^{\infty}} v(\xi_a)\bigg)
\end{align}
where we have sets $F_1:=C_0$ and 
$F_J:= C_0\cap C_1^c\cap\ldots\cap C_{J-1}^c$ 
for $J\ge 2$.

The equation \eqref{irvgDNLS} becomes
\begin{equation}
\label{NFReqstepJ}
\partial_t v =\Qcal(v)+  \sum_{j=2}^{J+1} \partial_t \Tcal_0^{(j)}(v) 
+ \sum_{j=2}^{J+1} \Tcal^{(j)}_{\Qcal}(v) 
+ \sum_{j=1}^{J} \Tcal^{(j)}_{\Tcal,1}(v) + \Tcal_{\Tcal}^{(J+1)}(v) .
\end{equation}

We record the formula for the term $\Tcal^{(J+1)}_{\Tcal,1}(v)$ appeared in the next step of NFR: 
\begin{align}
\label{defnTcalTcal1Jp1}
\Tcal^{(J+1)}_{\Tcal,1}(v)(\xi)&= \sum_{\TT\in\TF(J+1)} 
\int_{\bmxi\in\Xi_{\xi}(\TT)} 
  \ind_{F_J \cap C_J} \bigg(\prod_{j=1}^{J} \frac{e^{i {\mu}_{j}t } \xi^{(j)}_2}{\widetilde{\mu}_j} \bigg) 
   \big(e^{i{\mu}_{J+1} t} \xi^{(J+1)}_2 \big) \bigg(\prod_{a\in\TT^{\infty}} v(\xi_a)\bigg) \,,
\end{align}
where $F_J$ is defined above, and
\begin{equation}
\label{defn:CJ}
C_{J}=C_{J}(\xi; \TT) :=\Big\{ \bmxi \in\Xi_{\xi}(\TT) :  
\big|\widetilde{\mu}_{J} +\mu_{J+1} \big| \le \beta_{J} |\wt\mu_{J}|  \Big\}
\end{equation}
with $\beta_J\geq2$ to be determined later.

\subsection{The limit equation}

By iterating the normal form reduction step indefinitely, we formally derive the following limit equation: 
\begin{equation}
\label{limiteq}
\partial_t v = \Qcal(v) + \partial_t\bigg(\sum_{j=2}^{\infty} \Tcal_0^{(j)}(v) \bigg) 
	+ \sum_{j=2}^{\infty} \Tcal^{(j)}_{\Qcal}(v) 
	+ \sum_{j=1}^{\infty} \Tcal^{(j)}_{\Tcal,1}(v),
\end{equation}
where $\mathcal{T}_\Qcal^{(j)}$ and $\Tcal_{\Tcal,1}^{(j)}$ are $(2j+1)$-multilinear term, and $\Tcal_0^{(j)}$ is $(2j-1)$-multilinear term.
These multilinear terms $\mathcal{T}_\Qcal^{(j)}$, $\Tcal_{\Tcal,1}^{(j)}$, and $\Tcal_0^{(j)}$ appear as a result of $(j-1)$-many iterations of normal form reductions.

\section{The estimates in the strong norm}
\label{Sect:strongests}

We consider the trilinear operator 
$\Tcal_{\Phi}$  defined by 
\begin{equation}
\label{locmodNcal}
\F\Big[\Tcal_{\Phi}(v_1,v_2,v_3)\Big](t,\xi) = 
\int_{\xi=\xi_1-\xi_2+\xi_3} 
\frac{ |\xi_2|}{\jb{\Phi(\bxi)}^{\frac12}} 
\widehat{v_1}(\xi_1) \overline{\widehat{v_2}(\xi_2)} \widehat{v_3}(\xi_3) d\xi_1d\xi_2 \,,
\end{equation}
where $\Phi(\bxi)$ is given by \eqref{defn:Phi}. 
We can prove  the $H^s(\R)$-estimates  for all  higher order terms that appear in \eqref{limiteq} 
once we establish 
the following lemma:

\begin{lemma}[Basic trilinear estimate in the $H^s(\R)$-norm] 
\label{locmodest}
Let $s>\frac12$.  
Then there exists a finite constant $C=C(s)>0$ such that 
\begin{equation*}
\|\Tcal_{\Phi}(v_1,v_2,v_3)\|_{H_x^s(\R)} \leq C
 \prod_{j=1}^3 \|v_j\|_{H_x^s(\R)}   \,.
\end{equation*} 
\end{lemma}
\begin{proof}
By duality, the desired estimate  follows once we prove that
\begin{equation}
\label{eq:ME1dual}
\int_{\xi=\xi_1-\xi_2+\xi_3}
m(\overline{\xi}) 
\widehat{v_1}(\xi_1) \widehat{v_2}(\xi_2) \widehat{v_3}(\xi_3) \widehat{v_4}(\xi) d\xi_1d\xi_2d\xi 
	\leq C  
	  \prod_{j=1}^4 \|v_j\|_{L_x^2(\R)},
\end{equation}
for any  $v_1,\ldots, v_4\in L^2(\R)$ with $\widehat{v_j}\ge 0$ ($1\le j\le 4$), 
where 
the multiplier is given by
\begin{equation}
\label{defn:m}
m(\overline{\xi}) := 
 \frac{ |\xi_2|}{\jb{\Phi(\bxi)}^{\frac12}} \cdot \frac{\jb{\xi}^s }{\langle\xi_1\rangle^s \langle\xi_2\rangle^s \langle\xi_3\rangle^s}.
\end{equation}

\noindent 
\textbf{Case 1:} $\min(|\xi_2-\xi_1|,|\xi_2-\xi_3|)\leq 1$.

Without loss of generality, 
let us assume that $|\xi_2 - \xi_1|\leq1$. 
Since $\jb{\xi_1}\sim \jb{\xi_2}$ and $\jb{\xi_3}\sim \jb{\xi}$, we have $m(\bxi)\lesssim 1$. 
Denote $\zeta := \xi_2 - \xi_1 = \xi_3 - \xi$ and thus 
by using  H\"older's  inequality, we get that 
\begin{align*}
\text{LHS of }\eqref{eq:ME1dual}
& \leq \int_{|\zeta|\leq 1}  \int_{\xi_1}\widehat{v_1}(\xi_1) \widehat{v_2} (\xi_1 + \zeta) d\xi_1\int_{\xi_3} 
\widehat{v_3} (\xi_3) \widehat{v_4} (\xi_3 - \zeta)d\xi_3 \,d\zeta \\
& \leq  \bigg\|  \int_{\xi_1} \widehat{v_1}(\xi_1) \widehat{v_2} (\xi_1 + \zeta) d\xi_1\bigg\|_{L^\infty_\zeta}
\bigg\|\int_{\xi_3} \widehat{v_3} (\xi_3) \widehat{v_4} (\xi_3 - \zeta)d\xi_3 \bigg\|_{L^\infty_\zeta}\notag \\
& \lesssim \prod_{j = 1}^4 \|v_j\|_{L^2}. 
\end{align*}

\smallskip

For all of the remaining cases we assume that 
$|\xi_2-\xi_1|>1$ and $|\xi_2-\xi_3|>1$. 
Also, we note that the largest two frequencies necessarily have comparable sizes and that the multiplier $m$ 
is symmetric in $\xi_1, \xi_3$. 

We are using  the following known fact: 
\begin{equation}
\label{GVTtypeest}
\int_{\R} \frac{1}{\jb{\eta-\xi}^a \jb{\xi}^b}  \,d\xi \, \lesssim 1  \,,
\end{equation}
for any $a,b\ge 0$ such that $a+b>1$, with implicit constant independent of $\eta\in \R$. 
Indeed, this follows immediately from Young's convolution inequality:
$$ \big\| \big(\jb{\,\cdot\,}^{-a} * \jb{\,\cdot\,}^{-b}\big)(\eta) \big\|_{L^{\infty}_{\eta}(\R)} 
\leq \|\jb{\xi}^{-a}\|_{L^p_{\xi}(\R)} \|\jb{\xi}^{-b}\|_{L^q_{\xi}(\R)} \,,$$
with $p=\frac{a+b}{a}$ and $q=\frac{a+b}{b}$ 
(if $a$ or $b$ is zero, then \eqref{GVTtypeest} is trivially true).  

By Cauchy-Schwarz inequality (see, for example, \cite[Lemma 3.7]{Tao01}), 
for \eqref{eq:ME1dual},  
it is enough to show that
\begin{equation}
\label{eq:CS}
\mathcal{M}_j:=
\sup_{\xi_j\in\R}  \bigg(\int_{\xi=\xi_1-\xi_2+\xi_3} 
 m(\bxi)^2
 d\xi_k d\xi_{\ell} \bigg)^\frac{1}{2}
 \leq 
  C
\end{equation}
for some mutually distinct $1\leq  j, k, \ell \leq 4$ (with the convention that $\xi_4=\xi$). 
Indeed, by the Cauchy-Schwarz inequality with respect to $d\xi_k d\xi_\ell$ 
(with the index $r$ such that $\{j,k,\ell,r\}=\{1,2,3,4\}$), 
\begin{align*}
\textup{LHS of } \eqref{eq:ME1dual} &\leq \int_{\R} \left(\int_{\R^2} m(\bxi)^2 d\xi_k d\xi_\ell \right)^{\frac12}
 	\left(\int_{\R^2} \widehat{v_j}(\xi_j)^2 \widehat{v_k}(\xi_k)^2 
	\widehat{v_\ell}(\xi_\ell)^2 \widehat{v_r}(\xi_r)^2 d\xi_k d\xi_\ell \right)^{\frac12} d\xi_j  \\
	&\leq \mathcal{M}_j \int_{\R} \widehat{v_j}(\xi_j) 
	\left(\int_{\R^2}  \widehat{v_k}(\xi_k)^2 \widehat{v_\ell}(\xi_\ell)^2 
	\widehat{v_r}(\xi_r)^2 d\xi_k d\xi_\ell \right)^{\frac12} d\xi_j \\
	&\leq \mathcal{M}_j \|\widehat{v_j}\|_{L^2(\R)} 
	\left( \int_{\R} \int_{\R^2}  \widehat{v_k}(\xi_k)^2 \widehat{v_\ell}(\xi_\ell)^2 
	\widehat{v_r}(\xi_r)^2 d\xi_k d\xi_\ell d\xi_j  \right)^{\frac12}
\end{align*}
where in the last step we used the Cauchy-Schwarz inequality with respect to $d\xi_j$ 
and then \eqref{eq:ME1dual} follows from \eqref{eq:CS} by possibly  changing the order of integration on the right-hand side above 
(and taking into account the linear dependence $\xi_4=\xi_1-\xi_2+\xi_3$).  

Next, we discuss several cases 
based on the frequency size of the derivative factor 
$\partial_x \overline{v_2}$. 

\smallskip
\noi \textbf{Case 2:} $|\xi_2|^2\lesssim \jb{\Phi(\bxi)}$.

Since the largest two frequencies among $\xi$, $\xi_1$, $\xi_2$, and $\xi$ necessarily have comparable sizes, there exists at least one $\xi_i$, $1\leq i\leq 3$ such that $|\xi|\lesssim |\xi_i|$.
Without loss of generality, we assume that $|\xi|\lesssim |\xi_1|$.
In this case, we have
	\[m(\bxi) \lesssim \frac{1}{\jb{\xi_2}^{s}\jb{\xi_3}^{s}},\]
and
	\[\Mcal_4 \lesssim \sup_{\xi} \bigg(\int_{\xi=\xi_1-\xi_2+\xi_3} \frac{1}{\jb{\xi_2}^{2s}\jb{\xi_3}^{2s}}d\xi_{2}d\xi_{3}\bigg)^\frac{1}{2}\lesssim 1\]
for $s>\frac{1}{2}$.

\smallskip
\noi \textbf{Case 3:} $|\xi_2|^2\gg\jb{\Phi(\bxi)}$.

In this case, we have either $|\xi_2|\gg\jb{\xi_2-\xi_1}$ or $|\xi_2|\gg\jb{\xi_2-\xi_3}$.
It follows that either $|\xi_1|\sim|\xi_2|$ or $|\xi_2|\sim|\xi_3|$ must be hold.
Reminding that $\jb{\Phi(\bxi)}\sim|\xi_2|^2$ in the case when $|\xi_1|\sim|\xi_2|\sim|\xi_3|\gg|\xi|$, it is enough to treat following three subcases.

\smallskip
\noi \textbf{Subcase 3.a:} $|\xi_1|\sim|\xi_2|\gg|\xi_3|$.

In this case, we must have $|\xi_1|\sim|\xi_2|\gg|\xi|$, because $|\xi|\sim|\xi_1|$ implies that $\jb{\Phi(\bxi)} \sim\jb{\xi-\xi_3}\jb{\xi_2-\xi_3} \sim \jb{\xi_2}^2$.
If $|\xi|\lesssim |\xi_3|$, then we have
	\[m(\bxi) \lesssim \frac{1}{\jb{\xi_2-\xi_1}^\frac{1}{2}\jb{\xi_1}^{s}\jb{\xi_2}^{s-\frac{1}{2}}},\]
and
	\[\Mcal_4 \lesssim \sup_{\xi} \bigg\{\int_{\xi_1}\frac{1}{\jb{\xi_1}^{2s}}\bigg(\int_{|\xi_2-\xi_1|>1} \frac{1}{\jb{\xi_2-\xi_1}\jb{\xi_2}^{2s-1}}d\xi_2\bigg)d\xi_1\bigg\}^\frac{1}{2}\lesssim 1\]
for $s>\frac{1}{2}$ from \eqref{GVTtypeest}.

On the other hand, if $|\xi|\gg|\xi_3|$, then $\jb{\Phi(\bxi)} \sim \jb{\xi}\jb{\xi_2}$,
	\[m(\bxi) \sim \frac{\jb{\xi}^{s-\frac{1}{2}}}{\jb{\xi_2}^{2s-\frac{1}{2}}\jb{\xi_3}^{s}} \lesssim \frac{1}{\jb{\xi_2}^{s}\jb{\xi_3}^{s}}\]
for $s\geq\frac{1}{2}$, and
	\[\Mcal_4 \lesssim \sup_{\xi} \bigg(\int_{\xi=\xi_1-\xi_2+\xi_3} \frac{1}{\jb{\xi_2}^{2s}\jb{\xi_3}^{2s}}d\xi_{2}d\xi_{3}\bigg)^\frac{1}{2}\lesssim 1\]
whenever $s>\frac{1}{2}$.

\smallskip
\noi \textbf{Subcase 3.b.} $|\xi_2|\sim|\xi_3|\gg|\xi_1|$.

This case follows from Subcase 2.b. by switching $1 \leftrightarrow 3$.

\smallskip
\noi\textbf{Subcase 3.c:} $|\xi|\sim|\xi_1|\sim|\xi_2|\sim|\xi_3|$.

In this case,
	\[m(\bxi) \sim \frac{1}{\jb{\xi-\xi_3}^\frac{1}{2}\jb{\xi-\xi_1}^\frac{1}{2}\jb{\xi_1}^{s-\frac{1}{2}}\jb{\xi_3}^{s-\frac{1}{2}}}.\]
Hence, we have
	\begin{align*}
	\mathcal{M}_4& \lesssim \sup_{\xi} \bigg(\int_{\xi=\xi_1-\xi_2+\xi_3}\frac{1}{\jb{\xi-\xi_3}\jb{\xi-\xi_1}\jb{\xi_1}^{2s-1}\jb{\xi_3}^{2s-1}}d\xi_1d\xi_3\bigg)^\frac{1}{2}\\
	& \lesssim \sup_{\xi} \bigg\{\bigg(\int_{|\xi-\xi_1|>1} \frac{1}{\jb{\xi-\xi_1}\jb{\xi_1}^{2s-1}}d\xi_1\bigg)\bigg(\int_{|\xi-\xi_3|>1} \frac{1}{\jb{\xi-\xi_3}\jb{\xi_3}^{2s-1}}d\xi_3\bigg)\bigg\}^\frac{1}{2}\lesssim 1
	\end{align*}
for $s>\frac{1}{2}$ from \eqref{GVTtypeest}.
\end{proof}

\begin{remark}
By comparing the estimate of Lemma~\ref{locmodest}  
with the similar estimate for the cubic NLS on $\R$ (see \cite[Lemma~2.3]{KOY}), 
we note that whenever $m(\bar{\xi})\lesssim 1$ (e.g.  when $\min(|\xi_2-\xi_1|,|\xi_2-\xi_3|)\leq 1$ 
or when $|\xi_1|\sim|\xi_2|\sim |\xi_3|$), 
our operator $\Tcal_{\Phi}$ 
acts as the operator $\mathcal{N}^{0}_{\leq M}$ from \cite{KOY} 
(with displacement parameter $\alpha=0$ and localization size $M\sim 1$), 
and thus 
we can appeal to the arguments used therein.  
For the sake of completeness 
we have also included the argument for Case~1  in the proof of Lemma~\ref{locmodest} above. 
\end{remark}

\begin{remark}
\label{rmk:dxonconjug}
Notice that in the above proof, 
the case when $|\xi_2|\sim|\xi|\gg|\xi_1|,|\xi_3|$ in Case 2 informs us why 
the derivative falling on the conjugate factor in the cubic nonlinearity $v^2 \partial_x \overline{v}$ 
can be handled: in the worst case scenario of the low$\times$high$\times$low $\to$ high frequency interaction, 
we can use the $\frac12$-power of the modulation 
to cancel the factor $\xi_2$ in the numerator. 
This motivates the need to use the gauge transformation 
\eqref{gaugetr} to eliminate the nonlinearity $2|v|^2 \partial_x v$ from the right-hand side of \eqref{DNLS}. 
\end{remark}

\begin{remark}
At the end-point regularity $s=\frac12$, 
with minor changes in the proof, 
we can also obtain an estimate as in Lemma~\ref{locmodest}, but for $\Tcal_{\Phi}$ defined by 
\begin{equation}
\label{locmodNcal12}
\F\Big[\Tcal_{\Phi}(v_1,v_2,v_3)\Big](t,\xi) = 
\int_{\xi=\xi_1-\xi_2+\xi_3} 
\frac{ |\xi_2|}{\jb{\Phi(\bxi)}^{\frac12+\varepsilon}} 
\widehat{v_1}(\xi_1) \overline{\widehat{v_2}(\xi_2)} \widehat{v_3}(\xi_3) d\xi_1d\xi_2 \,,
\end{equation}
where $\varepsilon>0$ can be taken arbitrarily small. 
However, in this case $C=C(\varepsilon)\nearrow \infty$ as $\varepsilon\searrow 0$. 
This remark also applies to Corollaries~\ref{cor:firstTcalTcal1} and \ref{cor:estSF}, 
Lemmata~\ref{lem:Tcal0Jp1}, and \ref{lem:TcalQcalJp1}, 
but not to Lemma~\ref{lem:TcalTcal1Jp1}.
\end{remark}

In the proofs of the following lemmata, we freely use the Fourier lattice property of $H^s(\R)$, 
i.e. 
$$\big\|\F^{-1}\big( |\F(v)| \big)\big\|_{H^s(\R)} = \|v\|_{H^s(\R)} \,,$$ 
and thus we drop the modulus notation on factors such as $v(\xi)$ 
(which henceforth we assume to be non-negative).

%


\begin{corollary}
\label{cor:firstTcalTcal1}
Let $s>\frac12$. Then for $\Tcal_{\Tcal,1}^{(1)}=\Tcal_1(v)$ given by \eqref{firstdecomp}, we have
\begin{equation*}
\big\|\Tcal_{\Tcal, 1}^{(1)}(v) \big\|_{H_x^s(\R)} \lesssim 
 N^{\frac12} \|v\|^3_{H^s_x(\R)} \,.
 \end{equation*}
\end{corollary}
\begin{proof}
We have
\begin{align*}
\Big|\F\big[\Tcal_{\Tcal,1}^{(1)}(v)\big](\xi) \Big| 
&\le
\int_{\substack{\xi=\xi_1-\xi_2+\xi_3\\ |\Phi(\bar{\xi})|\le N}} 
N^{\frac12} N^{-\frac12}
|\xi_2|  \widehat v(t,\xi_1)  \widehat v(t,\xi_2)  \widehat v(t,\xi_3)  d\xi_1d\xi_2\\
&\lesssim N^{\frac12} \int_{\xi=\xi_1-\xi_2+\xi_3} \frac{|\xi_2|}{\jb{\Phi(\bxi)}^{\frac12}}  
\widehat v(t,\xi_1)  \widehat v(t,\xi_2)  \widehat v(t,\xi_3)  d\xi_1d\xi_2\\
&=N^{\frac12} \Big|\F\big[ \Tcal_{\Phi}(v)\big](\xi)\Big|
\end{align*}
and therefore the estimate follows from Lemma~\ref{locmodest}.
\end{proof}


For estimating the remaining nonlinear terms of \eqref{limiteq},  
it is convenient to introduce the mapping $\SF(\TT;\,\cdot\,)$ 
associated to an ordered tree $\TT$, say of generation $J$,  
which essentially applies the operator $\Tcal_{\Phi}$ iteratively taking into account the structure of $\TT$. 
We define these mappings by the following bottom-up algorithmic procedure. 

\begin{definition}
\label{DEF:S} 
Let $J\ge 1$ 
and $\TT\in  \TF(J)$. 
We define the $(2J+1)$-linear map $\SF(\TT;\,\cdot\,)$ 
on space-time functions $v_j \in C(I; H^s(\R))$ ($1\le j\le 2J+1=|\TT^{\infty}|$) 
by the following rules. 

\begin{itemize}

\item[(i)] Replace the $j$th terminal node of $\TT$ 
by $v_j$, for all $j\in\{1,\ldots, 2J+1\}$.

\item[(ii)] 
For $j=J, J-1, \ldots, 1$, 
replace the $j$th root node $r^{(j)}$  
by  the  trilinear operator $\Tcal_{\Phi}$ 
whose  arguments are given by 
 the functions associated with its three children. 

\end{itemize}
\end{definition}

For such mappings, 
we have  the following corollary which is a consequence of Lemma~\ref{locmodest}. 

\begin{corollary}
\label{cor:estSF}
Let $s>\frac12$, $J\ge 1$ and $\TT\in \TF(J)$.   
Then
\begin{equation*}
\big\|\SF(\TT;v_1,\ldots ,v_{2J+1}) \big\|_{H_x^s(\R)} \leq C^{J} 
 \prod_{j=1}^{2J+1} \|v_j\|_{H_x^s(\R)}   \,,
\end{equation*} 
where $C$ is the constant given by Lemma~\ref{locmodest}.  
\end{corollary}
\begin{proof}
It follows  immediately by successively applying Lemma~\ref{locmodest}. Namely, 
we start with the root node $r^{(1)}$ of $\TT$ and we move top-down on $\TT$. 
Since $\TT$ is a tree of generation $J$, 
it has $J$ many root nodes and thus we pick up the constant $C^J$. 
\end{proof}

Next, for simplicity we set $\beta_0:=1$ and for any $J\ge 1$ we put
\begin{equation}
\label{defn:bJ}
b_J:=  \prod_{j=0}^{J-1} \beta_j \,. 
\end{equation} 

\begin{remark}
\label{rmk:choicebetas}
For each $s>\frac{1}{2}$, we choose the constants $\beta_j$'s such that we ensure
$$ \sup_{J\ge 1} \frac{c_{J+1} \beta_J (10C)^{J+1} (2J+6)}{b_1^{\theta}\cdots b_{J-1}^{\theta}} 
\,\lesssim\, 1\,,$$
where $c_{J+1}=1\cdot 3\cdot 5\cdot\cdots\cdot (2J+1)$ (see \eqref{cj}) and $\theta=\theta(s):=\min\{2s-1,\frac{1}{2}\}$.
For instance, we may take
	\[\beta_j=(2j+3)^{\frac{2}{\theta}},\quad j\geq 1.\]
Then, one can observe that the factorial decay of denominator $5^{2J-2}\cdot 7^{2J-4}\cdot\cdots\cdot(2J-1)^4\cdot(2J+1)^2$ is enough to compensate the factorial growth term $c_{J+1}$ and the exponential growth term $(10C)^J$.
\end{remark}

We are now ready to prove the estimates for all nonlinear terms of \eqref{limiteq}, which we treat in decreasing order of difficulty.

\begin{lemma}
\label{lem:TcalTcal1Jp1}
Let $s>\frac{1}{2}$ and $J\ge 1$.  
Then, for $\Tcal_{\Tcal,1}^{(J+1)}$  given by \eqref{defnTcalTcal1Jp1} 
we have 
\begin{align}
\label{TcalTcal1Jp1}
\|\Tcal_{\Tcal,1}^{(J+1)}(v)\|_{H_x^s(\R)} &\lesssim N^{-\frac{1}{2}(J-1)} \|v\|_{H_x^s(\R)}^{2J+3}
\,,\\ 
\label{TcalTcal1Jp1diff}
\|\Tcal_{\Tcal,1}^{(J+1)}(v) -\Tcal_{\Tcal,1}^{(J+1)}(w) \|_{H_x^s(\R)} 
&\lesssim N^{-\frac{1}{2}(J-1)} \Big(\|v\|_{H_x^s(\R)}^{2J+2} + \|w\|_{H_x^s(\R)}^{2J+2} \Big) \|v-w\|_{H_x^s(\R)}\,.
\end{align}
\end{lemma}

\begin{proof}
With $\Tcal^{(J+1)}_{\Tcal,1} (\TT; v)$ simply denoting the summand in \eqref{defnTcalTcal1Jp1}, 
we have 
\begin{equation*}
\Tcal_{\Tcal,1}^{(J+1)}(v) = \sum_{\TT\in\TF(J+1)}  \Tcal^{(J+1)}_{\Tcal,1} (\TT; v) .
\end{equation*}
and thus
\begin{align}
\label{TcalTcal1Jsup}
\|\Tcal^{(J+1)}_{\Tcal,1} (v) \|_{H^s} 
& \leq  c_{J+1} \sup_{\TT\in \TF(J+1)} \| \Tcal^{(J+1)}_{\Tcal,1} (\TT; v)  \|_{H^s} \,.
\end{align}
Now fix $\TT\in \TF(J+1)$. 
We recall that the frequency support of $\Tcal_{\Tcal,1}^{(J+1)}(\TT;v)$  
is 
$$C_0\cap C_1^c \cap\cdots \cap C_{J-1}^c\cap C_J \,.$$ 
Hence, we have
$$|\mu_1|>N\,,\quad 
|\wt\mu_{j}| > \beta_{j-1} |\wt \mu_{j-1}| \text{ for }j=2,\ldots,J\,,\quad\text{and}\quad 
|\wt \mu_{J+1}|\le \beta_J | \wt \mu_{J}| \,.$$
In particular, $|\wt \mu_j|>b_j N$ for $j=1, \ldots, J$.  
Note that $\beta_{j-1}\ge 2$ for $j=2,\ldots, J$.
Then, from $|\mu_j| \leq|\wt\mu_j|+|\wt\mu_{j-1}| <\frac32 |\wt \mu_j|$ and 
$|\wt\mu_j| \leq |\mu_j|+|\wt\mu_{j-1}|<|\mu_j| + \frac12 |\wt \mu_j|$,
we deduce $ |\wt \mu_j| \sim |\mu_j|$, for $j=2,\ldots, J$. 
Also, since $|\mu_{J+1}| \le |\wt \mu_{J+1}| + |\wt \mu_J| \le (\beta_J+1) |\wt \mu_J|$, 
we get $|\mu_{J+1}| \leq 2\beta_J  |\wt \mu_J|$.
Thus we have 
\begin{align*}
\Tcal^{(J+1)}_{\Tcal,1} (\TT; v)  &\leq 
\int_{\bmxi\in\Xi_{\xi}(\TT)} 
  \ind_{C_J \cap F_J} \bigg(\prod_{j=1}^{J} \frac{ |\xi^{(j)}_2|}{|\widetilde{\mu}_j|} \bigg) 
   \big|\xi^{(J+1)}_2 \big| \bigg(\prod_{a\in\TT^{\infty}} v(\xi_a) \bigg)  \\
&\lesssim  \int_{\bmxi\in\Xi_{\xi}(\TT)} 
  \bigg(\prod_{j=1}^{J-1} \frac{ |\xi^{(j)}_2|}{ (b_j N)^{\frac12} \jb{\mu_j}^{\frac12}} \bigg) 
     \frac{|\xi_2^{(J)}| }{(2\beta_J)^{-\frac12} \jb{\mu_J}^{\frac12} \jb{\mu_{J+1}}^{\frac12}} \big|\xi^{(J+1)}_2 \big| 
     \bigg(\prod_{a\in\TT^{\infty}} v(\xi_a) \bigg)\\
&\lesssim \beta_J^{\frac12} \prod_{j=1}^{J-1} b_j^{-\frac12}  N^{-\frac12(J-1)} \cdot \SF(\TT;v)
\end{align*} 
Therefore, by Corollary~\ref{cor:estSF} and \eqref{TcalTcal1Jsup}, 
we get 
$$ \| \Tcal^{(J+1)}_{\Tcal,1} (v)  \|_{H^s(\R)}  \lesssim 
\frac{c_{J+1} \beta_J^{\frac12} C^{J+1}}{b_1^{\frac12}\cdots b_{J-1}^{\frac12}}  
N^{-\frac12(J-1)} \|v\|_{H^s(\R)}^{2J+3} \,.$$

For the difference estimate \eqref{TcalTcal1Jp1diff}, 
a similar argument applies. 
Namely, one writes the difference using a telescopic sum and employs the multilinear version of the operator 
$\SF(T,\cdot)$ with precisely one entry being $v-w$ and the others being either $v$ or $w$. 
Compared to \eqref{TcalTcal1Jp1}, 
we note that for \eqref{TcalTcal1Jp1diff} we pick up an extra factor of $2J+4$ since we have the bound 
$$\big|a^{2J+3} - b^{2J+3}\big| \leq \bigg(\sum_{j=1}^{2J+3} |a|^{2J+3-j} |b|^{j-1}\bigg)|a-b|\leq 
(2J+4)\Big(|a|^{2J+2} +|b|^{2J+2}\Big)|a-b|\,.$$ 
Hence, 
$$ \| \Tcal^{(J+1)}_{\Tcal,1} (v) -  \Tcal^{(J+1)}_{\Tcal,1} (w)\|_{H^s(\R)}  \lesssim 
\frac{c_{J+1} \beta_J^{\frac12} C^{J+1} (2J+4)}{b_1^{\frac12}\cdots b_{J-1}^{\frac12}} 
\Big(\|v\|_{H^s(\R)}^{2J+2}  + \|w\|_{H^s(\R)}^{2J+2} \Big)
\|v-w\|_{H^s(\R)}^{2J} \,.$$
By taking into account Remark~\ref{rmk:choicebetas} we deduce \eqref{TcalTcal1Jp1} and \eqref{TcalTcal1Jp1diff}. 
\end{proof}


Next, we consider the nonlinear terms coming as boundary terms when applying integration by parts with respect to the temporal variable in Section~\ref{sect2}.

\begin{lemma}
\label{lem:Tcal0Jp1}
Let $s >\frac12$  and $J\ge 1$.  
Then, for $\Tcal_0^{(J+1)}$  given by \eqref{defnTcal0Jp1} 
we have 
\begin{align}
\label{estTcal0Jp1inHs}
\|\Tcal_0^{(J+1)}(v)\|_{H^s(\R)} &\lesssim N^{-\frac{1}{2}J} \|v\|_{H^s(\R)}^{2J+1}
\,,\\ 
\label{estTcal0Jp1diff}
\|\Tcal_0^{(J+1)}(v) -\Tcal_0^{(J+1)}(w) \|_{H^s(\R)} &\lesssim N^{-\frac{1}{2}J} 
 \Big(\|v\|_{H^s(\R)}^{2J} + \|w\|_{H^s(\R)}^{2J} \Big) \|v-w\|_{H^s(\R)} \,.
\end{align}
\end{lemma}

\begin{proof}
With $\Tcal^{(J+1)}_{0} (\TT; v)$ simply denoting the summand in \eqref{defnTcal0Jp1}, 
we have 
\begin{equation}
\label{Tcal0SF0}
\Tcal_0^{(J+1)}(v) = \sum_{\TT\in\TF(J)}  \Tcal_0^{(J+1)} (\TT; v) .
\end{equation}
and thus
\begin{align}
\label{Tcal0Jsup}
\|\Tcal^{(J+1)}_0 (v) \|_{H^s} 
& \leq  c_J \sup_{\TT\in \TF(J)} \| \Tcal_0^{(J+1)}( \TT;  v)\|_{H^s} \,.
\end{align}
Now fix $\TT\in \TF(J)$. 
We recall that the frequency support of $\Tcal_{0}^{(J+1)}(\TT;v)$  
is 
$F_J=C_0\cap C_1^c \cap\cdots \cap C_{J-1}^c$. 
Hence, we have
$|\mu_1|>N\,,\  
|\wt\mu_{j}| > \beta_{j-1} |\wt \mu_{j-1}| \text{ for $j=2,\ldots,J$}$. 
As in the proof of Lemma~\ref{lem:TcalTcal1Jp1}, 
we have $|\mu_j|\sim |\wt \mu_j|>b_{j-1} N$ for $j=2, \ldots, J$.  
Thus we have 
\begin{align*}
\Tcal^{(J+1)}_{0} (\TT; v)  &\leq 
\int_{\bmxi\in\Xi_{\xi}(\TT)} 
  \ind_{F_J} \bigg(\prod_{j=1}^{J} \frac{ |\xi^{(j)}_2|}{|\widetilde{\mu}_j|} \bigg) 
  \bigg(\prod_{a\in\TT^{\infty}} v(\xi_a) \bigg)  \\
&\lesssim  \int_{\bmxi\in\Xi_{\xi}(\TT)} 
  \bigg(\prod_{j=1}^{J} \frac{ |\xi^{(j)}_2|}{ (b_j N)^{\frac12} \jb{\mu_j}^{\frac12}} \bigg) 
          \bigg(\prod_{a\in\TT^{\infty}} v(\xi_a) \bigg)\\
&\lesssim \bigg( \prod_{j=1}^{J-1} b_j^{-\frac12} \bigg) N^{-\frac12 J} \cdot \SF(\TT;v)
\end{align*} 
Therefore, by Corollary~\ref{cor:estSF} and \eqref{Tcal0Jsup}, 
we get 
$$ \| \Tcal^{(J+1)}_{0} (v)  \|_{H^s(\R)}  \lesssim 
\frac{c_{J}  C^J}{b_1^{\frac12}\cdots b_{J-1}^{\frac12}}  N^{-\frac12 J} \|v\|_{H^s(\R)}^{2J+1} \,.$$

For the difference estimate \eqref{estTcal0Jp1diff}, 
an observation analogous to that in the proof of Lemma~\ref{lem:TcalTcal1Jp1} applies
and thus we obtain
$$\|\Tcal_0^{(J+1)}(v) -\Tcal_0^{(J+1)}(w) \|_{H^s(\R)} \lesssim 
\frac{c_{J}  C^J(2J+2)}{b_1^{\frac12}\cdots b_{J-1}^{\frac12}} N^{-\frac{1}{2}J} \Big(\|v\|_{H^s(\R)}^{2J}  + \|w\|_{H^s(\R)}^{2J} \Big) \|v-w\|_{H^s(\R)} \,.$$
\end{proof}

In the proofs  of the following lemma, 
we skip the argument for the difference estimate altogether as the same ideas apply as for the difference estimate of Lemma~\ref{lem:Tcal0Jp1}.

\begin{lemma}
\label{lem:TcalQcalJp1}
Let $s>\frac12$ and $J\ge 1$.  
Then, for $\Tcal_{\Qcal}^{(J+1)}$  given by \eqref{defnTcalQcalJp1} 
we have 
\begin{align}
\label{estTcalQcalJp1}
\|\Tcal_{\Qcal}^{(J+1)}(v)\|_{H_x^s(\R)} &\lesssim N^{-\frac{1}{2}J}\|v\|_{H_x^s(\R)}^{2J+5}
,\\ 
\label{estTcalQcalJp1diff}
\|\Tcal_{\Qcal}^{(J+1)}(v) -\Tcal_{\Qcal}^{(J+1)}(w) \|_{H_x^s(\R)} &\lesssim N^{-\frac{1}{2}J} 
	\Big( \|v\|_{H_x^s(\R)}^{2J+4}  + \|w\|_{H_x^s(\R)}^{2J+4} \Big) \|v-w\|_{H_x^s(\R)}
\end{align}
\end{lemma}

\begin{proof}
The proof is similar to the proof of Lemma~\ref{lem:Tcal0Jp1}. 
We have 
\begin{align}
\label{TcalQcalJsup}
\big\|\Tcal^{(J+1)}_{\Qcal} (v) \big\|_{H^s} 
& \leq  c_J(2J+1) \sup_{\TT\in \TF(J)}  \sup_{b\in T^{\infty}} \big\| \Tcal^{(J+1)}_{\Qcal} ( \TT,b;  v) \big\|_{H^s} \,,
\end{align}
where $\Tcal^{(J+1)}_{\Qcal} ( \TT,b;  v)$ denotes the (inner-most) summand in \eqref{defnTcalQcalJp1}. 
Fix $\TT\in \TF(J)$ and $b\in T^{\infty}$. 
Then we have 
\begin{align*}
\Tcal^{(J+1)}_{\Qcal} ( \TT,b;  v)   &\leq 
\int_{\bmxi\in\Xi_{\xi}(\TT)} 
  \ind_{F_J} \bigg(\prod_{j=1}^{J} \frac{ |\xi^{(j)}_2|}{|\widetilde{\mu}_j|} \bigg) 
 \bigg(\Qcal(v)(\xi_b) \prod_{\substack{a\in\TT^{\infty}\\ a\neq b}} v(\xi_a)\bigg) \\
&\lesssim  \int_{\bmxi\in\Xi_{\xi}(\TT)} 
  \bigg(\prod_{j=1}^{J} \frac{ |\xi^{(j)}_2|}{ (b_j N)^{\frac12} \jb{\mu_j}^{\frac12}} \bigg) 
          \bigg(\Qcal(v)(\xi_b) \prod_{\substack{a\in\TT^{\infty}\\ a\neq b}} v(\xi_a)\bigg)\\
&\lesssim \bigg( \prod_{j=1}^{J-1} b_j^{-\frac12} \bigg) N^{-\frac12 J} \cdot \SF(\TT; \mathbf{v}_b) \,,
\end{align*} 
where if $b$ is the $j$th terminal node of $\TT$, we put 
$$\mathbf{v}_b:=(v,\ldots,v,\underbrace{\Qcal(v)}_{j\text{th spot}},v,\ldots,v) \,.$$
Therefore, by Corollary~\ref{cor:estSF},  \eqref{NfrakHsest},  and \eqref{TcalQcalJsup},
\begin{align*}
\big\| \Tcal^{(J+1)}_{\Qcal} ( v) \big\|_{H_x^s(\R)} 
& \lesssim 
\frac{c_{J}  C^J (2J+1)}{b_1^{\frac12}b_2^\frac{1}{2}\cdots b_{J-1}^{\frac12}}  N^{-\frac12 J} \|v\|_{H^s(\R)}^{2J+5} 
\end{align*}
For the difference estimate \eqref{estTcalQcalJp1diff}, 
an observation analogous to that in the proof of Lemma~\ref{lem:TcalTcal1Jp1} 
(see also the proof of Lemma~\ref{lem:Tcal0Jp1}) applies and 
we  take into account Remark~\ref{rmk:choicebetas}. 
\end{proof}

\section{The estimates in a weak norm}
\label{sect:weakests}

Here, we prove the estimates necessary to rigorously justify 
the normal form equation \eqref{limiteq} for rough $H^s(\R)$-solutions of 
\eqref{irvgDNLS}, which is done explicitly in Section~\ref{sect:justif}. 
For this purpose, we have to be able to estimate $\partial_t v$, for $v\in C(I; H^s(\R))$ solution to \eqref{irvgDNLS}. 

It is clear that due to the derivative in the cubic nonlinearity, the estimate 
$$\big\|v^2 \partial_x \overline{v}\big\|_{H_x^s(\R)} \lesssim \|v\|^3_{H^s_x(\R)}$$
fails. 
However, if we weaken the norm in the left-hand side above, 
then we might be able to obtain an estimate satisfactory to our aims in Section~\ref{sect:justif}. 
Hence, with the following lemma, we identify a family of Sobolev norms weaker than the $H^s(\R)$-norm 
which can serve as a weak topology used to justify the normal form equation \eqref{limiteq}. 


\begin{lemma}
\label{estTcalvHsm1}
Let $s>\frac12$ and $\sigma\le s-1$. Then, we have the trilinear estimate
\begin{equation*}
\big\| v_1\big( \partial_x \overline{v_2} \big) v_3\big\|_{H^{\sigma}_x(\R)} \lesssim \prod_{j=1}^3 \|v_j\|_{H_x^s(\R)} \,.
\end{equation*}
\end{lemma}
\begin{proof}
By duality, the desired estimate follows once we show: 
\begin{equation}
\label{wed:cubicderiv}
 \int_{\xi=\xi_1-\xi_2+\xi_3} 
 m_4(\bar{\xi}) 
v_1(\xi_1) v_2(\xi_2) v_3(\xi_3) v_4(\xi)  d\xi_1d\xi_2d\xi 
\lesssim \prod_{k=1}^4 \|u_k\|_{L_{\xi}^2(\R)} \,,
\end{equation}
for any  $v_1,\ldots, v_4\in L^2(\R)$ with $\widehat{v_j}\ge 0$ ($1\le j\le 4$), 
with the multiplier 
\begin{equation}
m_4(\bar{\xi}) = \frac{\jb{\xi}^{\sigma} |\xi_2|}{\jb{\xi_1}^s \jb{\xi_2}^s \jb{\xi_3}^s} \,.
\end{equation}
We study the boundedness of this multiplier, 
distinguishing which two of the four frequencies are the largest. 
On the convolution hyperplane, 
it must be that the largest two frequencies are comparable.  
Also, by the symmetry of $m_4$ with respect to $\xi_1, \xi_3$, 
we may assume without loss of generality that 
$|\xi_1|\geq|\xi_3|$. 

\smallskip
\noindent
\textbf{Case 1:} $|\xi|\sim |\xi_2|\gtrsim |\xi_1|, |\xi_3|$. 

In this case, since $\sigma+ 1-s\le 0$, we have
$$m_4(\bar{\xi}) \lesssim  \jb{\xi}^{\sigma+1-s}   \frac{1}{\jb{\xi_1}^s \jb{\xi_3}^s} 
 \leq \frac{1}{\jb{\xi_1}^s \jb{\xi_3}^s}\,.$$ 

\noindent
\textbf{Case 2:} $|\xi|\sim |\xi_1|\gtrsim |\xi_2|, |\xi_3|$.

Since $\sigma+ 1-s\le 0$, we have
$$m_4(\bar{\xi}) \lesssim  \jb{\xi}^{\sigma+1-s} \frac{\jb{\xi_2}^{1-s}}{\jb{\xi} \jb{\xi_3}^s} 
 \lesssim \frac{1}{\jb{\xi_2}^s \jb{\xi_3}^s}\,.$$ 

\noindent
\textbf{Case 3:} $|\xi_2|\sim |\xi_1|\gtrsim |\xi|, |\xi_3|$.

Since $s+\sigma\leq 2s-1$, we have$\jb{\xi}^{s+\sigma}\leq \jb{\xi}^{2s-1}\lesssim\jb{\xi_2}^{2s-1}$ for $s\geq\frac{1}{2}$, and
$$m_4(\bar{\xi}) \lesssim \frac{\jb{\xi}^\sigma}{\jb{\xi_2}^{2s-1}\jb{\xi_3}^s}
  \lesssim \frac{1}{\jb{\xi}^s \jb{\xi_3}^s}\,.$$

\noindent
\textbf{Case 4:} $|\xi_1|\sim |\xi_3|\gtrsim |\xi|, |\xi_2|$.

Since $s+\sigma\leq 2s-1$, we have $\jb{\xi_2}\jb{\xi}^{s+\sigma}\leq \jb{\xi_2}\jb{\xi}^{2s-1}\lesssim \jb{\xi_1}^{2s}$ for $s\geq\frac{1}{2}$, and
$$m_4(\bar{\xi}) 
 \lesssim \frac{\jb{\xi_2}^{1-s}\jb{\xi}^\sigma}{\jb{\xi_1}^{2s}}\lesssim\frac{1}{\jb{\xi}^s \jb{\xi_2}^s}\,.$$ 

\smallskip
In each of the four cases, 
there exist $k_1$, $k_2\in \{1,2,3,4\}$ , $k_1\neq k_2$ such that 
$$m_4(\bar{\xi}) \lesssim \frac{1}{\jb{\xi_{k_1}}^{\frac12+} \jb{\xi_{k_2}}^{\frac12+}}$$ 
(with the convention that $\xi_4=\xi$) and let $j$ denote the third index .
Then, by Cauchy-Schwarz inequality, 
the Sobolev embedding $H^s\hookrightarrow L^{\infty}$, and the fact that $H^s(\R)$ is a Fourier lattice, we have 
\begin{align*}
\textup{LHS of }\eqref{wed:cubicderiv} \lesssim 
 \prod_{k\in\{k_1,k_2\}} \left\|\jb{\partial_x}^{-s}\mathcal{F}^{-1}[|u_k|] \right\|_{L_x^{\infty}} 
 \|u_j\|_{L_{\xi}^2} \|u_4\|_{L^2_{\xi}}
 \lesssim \prod_{k=1}^4 \|u_k\|_{L^2_{\xi}} 
\end{align*}
and the proof is completed. 
\end{proof}

As a consequence of Lemma~\ref{estTcalvHsm1} and \eqref{NfrakHsest}, we have the following: 

\begin{corollary}
\label{lem:dtvest}
Let $s>\frac12$ and $v\in C(I; H^s(\R))$ be a solution to \eqref{irvgDNLS}. 
Then, uniformly in $t\in I$, we have
\begin{equation}
\|\partial_t v\|_{H^{s-1}_x(\R)} \lesssim \|v\|^3_{H^{s}_x(\R)} + \|v\|^5_{H^{s}_x(\R)} \,.
\end{equation}
\end{corollary}


Next, for $M\geq1$, we consider the trilinear operator 
$\Tcal^{\textup{w}}_{|\Phi| >M}$  defined by 
\begin{equation}
\label{locmodNcalw}
\F\Big[\Tcal^{\textup{w}}_{|\Phi|>M}(v_1,v_2,v_3)\Big](t,\xi) = 
\int_{\substack{\xi=\xi_1-\xi_2+\xi_3\\ |\Phi(\bxi)| > M} }
\frac{ |\xi_2|}{\jb{\Phi(\bxi)} } 
\widehat{v_1}(\xi_1) \overline{\widehat{v_2}(\xi_2)} \widehat{v_3}(\xi_3) d\xi_1d\xi_2 \,,
\end{equation}
where $\Phi(\bxi)$ is given by \eqref{defn:Phi}.

\begin{lemma}[The estimate of $\Tcal^{\textup{w}}_{|\Phi|>M}$ in the $H^{s-1}(\R)$-norm]
\label{lem:WE2}
Let $s>\frac{1}{2}$ and $\theta=\theta(s)\colonequals\min\{2s-1,\frac{1}{2}\}$. 
Then, 
there exists a finite constant $C=C(s)>0$ such that 
\begin{equation*}
\|\Tcal^{\textup{w}}_{|\Phi| > M}(v_1,v_2,v_3)\|_{H^{s-1}(\R)} \leq C 
   M^{-\theta}
 \|v_j\|_{H^{s-1}(\R)} 
  \|v_k\|_{H^s(\R)} \|v_{l}\|_{H^s(\R)}, 
\end{equation*}
for any $j,k,\ell$ such that $\{j,k,l\}=\{1,2,3\}$ and for any $M\ge 1$. 
\end{lemma}
\begin{proof}
Similarly to the proof of Lemma~\ref{locmodest}, 
it suffices to prove that
\begin{align}
\label{eq:MEjdual}
\int_{\substack{\xi=\xi_1-\xi_2+\xi_3\\ |\Phi(\bxi)| > M} } 
m_j(\overline{\xi}) 
 \widehat{v_1}(\xi_1) \widehat{v_2}(\xi_2) \widehat{v_3}(\xi_3) \widehat{v_4}(\xi) \,d\xi_1d\xi_2d\xi 
 \,\leq\, 
	C M^{-\theta} \prod_{j=1}^4 \|v_j\|_{L_x^2} \,,
\end{align}
for any  $v_1,\ldots, v_4\in L^2(\R)$ with $\widehat{v_j}\ge 0$ ($1\le j\le 4$). 
Also, by Cauchy-Schwarz inequality, it suffices to check that 
\begin{equation*}
\mathcal{M}^j_{k}:=
\sup_{\xi_k\in\R} 
\Bigg(
\int_{\substack{\xi=\xi_1-\xi_2+\xi_3\\ |\Phi(\bxi)| > M} } 
 m_j(\bxi)^2
 d\xi_{\ell_1} d\xi_{\ell_2} 
 \Bigg)^{\frac12}
 \leq 
 C M^{-\theta}\,,
\end{equation*}
for some $1\le k\le 4$, 
where 
the multiplier  is given by 
\begin{equation}
\label{defn:mj}
m_j(\bxi):=
 \frac{ |\xi_2|}{\jb{\Phi(\bxi)}} \cdot \frac{\jb{\xi_j}^{1-s} }{\jb{\xi}^{1-s}\jb{\xi_k}^s\jb{\xi_{\ell}}^s} 
 	 =  \frac{\jb{\xi_j}}{\jb{\xi} \jb{\Phi(\bxi)}^{\frac12} } m(\bxi)
\end{equation}
with $\{j,k,\ell\}=\{1,2,3\}$ and $m(\bxi)$ given by \eqref{defn:m}.

\smallskip 

Let us first prove the lemma for  $j=1$.

\smallskip 
\noi \textbf{Case 1:} $\min(|\xi_2-\xi_1|, |\xi_2-\xi_3|)\le 1$.

Since $m_1$ is not symmetric in $\xi_1$, $\xi_3$,
we  treat the following two subcases. 

\smallskip
\noindent
\textbf{Subcase 1.1:} $|\xi_2-\xi_1|\le 1$. Then $\jb{\xi_1}\sim \jb{\xi_2}$ and also $\jb{\xi_3}\sim \jb{\xi}$. 
We have 
$$m_1(\bxi) \sim \frac{|\xi_2|}{\jb{\Phi(\bxi)} \jb{\xi_3} \jb{\xi_1}^{2s-1}}
\lesssim \frac{|\xi_2|^{2-2s}}{\jb{\Phi(\bxi)} }\,.$$

Assume for now that $|\xi_2|\gg \jb{\xi_3}$. 
Then $\jb{\Phi(\bxi)} \sim \jb{\xi_2 (\xi_2-\xi_1)} $ and thus 
$$m_1(\bxi) \lesssim 
\frac{|\xi_2(\xi_2-\xi_1)|^{2-2s}}{ \jb{\Phi(\bxi)}^{\gamma} \jb{\xi_2(\xi_2-\xi_1)}^{2-2s} |\xi_2-\xi_1|^{2-2s}}
\lesssim \frac{M^{-(2s-1)}}{|\xi_2-\xi_1|^{1-(2s-1)}} \,.$$
Similarly to Case~1 in the proof of Lemma~\ref{locmodest}, 
we denote $\zeta := \xi_2 - \xi_1 = \xi_3 - \xi$ and 
by using  H\"older's  inequality, we get that 
\begin{align*}
\text{LHS of }\eqref{eq:MEjdual}
& \lesssim \int_{|\zeta|\leq 1} \frac{M^{-(2s-1)}}{|\zeta|^{1-(2s-1)}} 
  \int_{\xi_1}\widehat{v_1}(\xi_1) \widehat{v_2} (\xi_1 + \zeta) d\xi_1\int_{\xi_3} 
\widehat{v_3} (\xi_3) \widehat{v_4} (\xi_3 - \zeta)d\xi_3 \,d\zeta \\
& \leq \left(\int_{|\zeta|\leq 1} \frac{M^{-(2s-1)}  d\zeta}{|\zeta|^{1-(2s-1)}} \right) \bigg\|  \int_{\xi_1} \widehat{v_1}(\xi_1) \widehat{v_2} (\xi_1 + \zeta) d\xi_1\bigg\|_{L^\infty_\zeta}
\bigg\|\int_{\xi_3} \widehat{v_3} (\xi_3) \widehat{v_4} (\xi_3 - \zeta)d\xi_3 \bigg\|_{L^\infty_\zeta}\notag \\
& \lesssim  M^{-(1-2s)}   \prod_{j = 1}^4 \|v_j\|_{L^2} \,. 
\end{align*}

If $|\xi_2|\lesssim \jb{\xi_3}$, then  $m_1(\bxi)\lesssim M^{-1}$ 
and in the argument above we use $\int_{|\zeta|\le 1} d\zeta \lesssim 1$. 

\smallskip
\noindent
\textbf{Subcase 1.2:} $|\xi_2-\xi_3|\le 1$. 
Then $\jb{\xi_2}\sim \jb{\xi_3}$ and also $\jb{\xi_1}\sim \jb{\xi}$. We have 
$$m_1(\bxi) \sim \frac{|\xi_2|}{\jb{\Phi(\bxi)} \jb{\xi_2}^{2s}} \lesssim M^{-1} $$
and we argue as in Subcase 1.1 above. 

\smallskip
In all the cases below, we assume that $|\xi_2-\xi_1|>1$ and $|\xi_2-\xi_3|>1$.
If $|\xi_1|\lesssim |\xi|$, then from \eqref{defn:mj} and condition $|\Phi(\bxi)|\leq M$, we can see that
	\[m_1(\bxi)\lesssim M^{-\frac{1}{2}}m(\bxi),\]
where $m(\bxi)$ is given by \eqref{defn:m}.
From Case 2 in the proof of Lemma~\ref{locmodest}, we have
	\[\mathcal{M}_4^1\lesssim M^{-\frac{1}{2}}\]
for $s>\frac{1}{2}$.
So we may assume that $|\xi_1|\gg|\xi|$.
In the case when $s\geq 1$, we have $\jb{\Phi(\bxi)}\sim \jb{\xi_1}\jb{\xi_2-\xi_1}$, and
	\[m_1(\bxi)\lesssim \frac{1}{\jb{\Phi(\bxi)}^\frac{1}{2}}\cdot\frac{1}{\jb{\xi_1}^\frac{1}{2}\jb{\xi_2-\xi_1}^\frac{1}{2}\jb{\xi_3}^s}\leq \frac{M^{-\frac{1}{2}}}{\jb{\xi_1}^\frac{1}{2}\jb{\xi_2-\xi_1}^\frac{1}{2}\jb{\xi_3}^s}.\]
Thus, from \eqref{GVTtypeest}, we have
	\[\Mcal_{2}^1\lesssim M^{-\frac{1}{2}} \sup_{\xi_2\in\R}\bigg(\int_{\xi_3}\frac{1}{\jb{\xi_3}^{2s}}d\xi_3\int_{|\xi_2-\xi_1|>1}\frac{1}{\jb{\xi_1}\jb{\xi_2-\xi_1}}d\xi_1\bigg)^\frac{1}{2}\lesssim M^{-\frac{1}{2}}\]
for $s>\frac{1}{2}$.

\smallskip
For the case when $\frac{1}{2}<s<1$, we further consider the following two cases.

\smallskip
\noi  \textbf{Case 2:} $\max\{ \jb{\xi_1}^2, \jb{\xi_2}^2\}\lesssim \jb{\Phi(\bxi)}$.

In this case,
	\[m_1(\bxi) \lesssim  \frac{1}{\jb{\Phi(\bxi)}^\frac{s}{2}}\cdot \frac{|\xi_2|\jb{\xi_1}^{1-s}}{\max\{\jb{\xi_1}^2,\jb{\xi_2}^2\}^{1-\frac{s}{2}}\jb{\xi}^{1-s}} \cdot \frac{1}{\jb{\xi_2}^s\jb{\xi_{3}}^s} \leq \frac{M^{-\frac{s}{2}}}{\jb{\xi_2}^s\jb{\xi_3}^s}\]
and thus
	\[\Mcal_4^1 \lesssim M^{-\frac{s}{2}}\sup_{\xi} \bigg(\int_{\xi=\xi_1-\xi_2+\xi_3} \frac{d\xi_{2}d\xi_{3}}{\jb{\xi_2}^{2s}\jb{\xi_3}^{2s}}\bigg)^\frac{1}{2}\lesssim M^{-\frac{s}{2}}.\]

\smallskip
\noi  \textbf{Case 3:} $\max\{ \jb{\xi_1}^2, \jb{\xi_2}^2\}\gg \jb{\Phi(\bxi)}$.
	
By arguing as Case 2 in Lemma~\ref{locmodest}, it is enough to treat following two subcases.

\smallskip
\noi \textbf{Subase 3.a:} $|\xi_1|\sim|\xi_2|\gg |\xi|, |\xi_3|$.

In this case, we have $\jb{\Phi(\bxi)}\sim \jb{\xi_1}\jb{\xi-\xi_3}$.
When $\frac{1}{2} < s < \frac{3}{4}$,
	\[m_1(\bxi) \sim \frac{1}{\jb{\Phi(\bxi)}^{2s-1}}\cdot \frac{1}{\jb{\xi-\xi_3}^{2-2s}\jb{\xi_3}^s\jb{\xi}^{1-s}}\leq \frac{M^{-(2s-1)}}{\jb{\xi-\xi_3}^{2-2s}\jb{\xi_3}^s\jb{\xi}^{1-s}}\]
Thus, from \eqref{GVTtypeest}, we have
	\[\Mcal_1^1 \lesssim M^{-(2s-1)} \sup_{\xi_1} \bigg\{\int_{\xi_3}\frac{1}{\jb{\xi_3}^{2s}}\bigg(\int_{|\xi-\xi_3|>1} \frac{1}{\jb{\xi-\xi_3}^{4-4s}\jb{\xi}^{2-2s}}d\xi \bigg)d\xi_3\bigg\}^\frac{1}{2}\lesssim M^{-(2s-1)}.\]
On the other hand, if $\frac{3}{4}\leq s<1$, then 
	\[m_1(\bxi) \sim \frac{1}{\jb{\Phi(\bxi)}^{\frac{1}{2}}}\cdot \frac{1}{\jb{\xi_1}^{2s-\frac{3}{2}}\jb{\xi-\xi_3}^{\frac{1}{2}}\jb{\xi_3}^s\jb{\xi}^{1-s}}\leq \frac{M^{-\frac{1}{2}}}{\jb{\xi-\xi_3}^{\frac{1}{2}}\jb{\xi_3}^s\jb{\xi}^{s-\frac{1}{2}}}\]
and thus
	\[\Mcal_1^1 \lesssim M^{-\frac{1}{2}}\sup_{\xi_1}  \bigg\{\int_{\xi_3}\frac{1}{\jb{\xi_3}^{2s}}\bigg(\int_{|\xi-\xi_3|>1} \frac{1}{\jb{\xi-\xi_3}\jb{\xi}^{2s-1}}d\xi \bigg) d\xi_3\bigg\}^\frac{1}{2}\lesssim M^{-\frac{1}{2}}.\]
%

\smallskip
\noi \textbf{Subcase 3.b.} $|\xi_2|\sim|\xi_3|\gg|\xi_1|\gg|\xi|$.

This case follows from Subcase 3.a. by switching $1 \leftrightarrow 3$.

\smallskip
This finishes the proof for $j=1$. Notice that the case $j=3$ is symmetric to the case $j=1$. 
It remains to discuss the case $j=2$. In this case, by the symmetry of $m_2$ with respect to $\xi_1, \xi_3$, 
we may assume without loss of generality that 
$|\xi_1|\geq|\xi_3|$. 
If $\jb{\xi_2}\lesssim \jb{\xi_1}$, then it is easy to check that $m_2(\bxi) \lesssim m_1(\bxi)$ and 
thus \eqref{eq:MEjdual} for $j=2$ follows from \eqref{eq:MEjdual} for $j=1$. 

Now, let us assume that $j=2$ and that $\jb{\xi_2}\gg \jb{\xi_1}$. 
In fact, in this case, we have $\jb{\xi}\sim \jb{\xi_2}\gg \jb{\xi_1} \ge \jb{\xi_3}$ 
which implies $\jb{\Phi(\bxi)} \sim \jb{\xi_2}^2$ and  
$$m_2(\bxi) \sim  \frac{1}{ \jb{\Phi(\bxi)}^{\frac12}} \cdot \frac{1}{\jb{\xi_1}^s \jb{\xi_3}^s} 
\lesssim \frac{M^{-\frac12}}{\jb{\xi_1}^s \jb{\xi_3}^s}$$
which is square integrable on $(\R^2,d\xi_1 d\xi_3)$ for $s>\frac{1}{2}$. 
This concludes the proof of Lemma~\ref{lem:WE2} for all three possible values of $j$.  
\end{proof}

\begin{lemma}[The estimate of $\Tcal^{\textup{w}}_{|\Phi|>M}$ in the $H^{s}(\R)$-norm] 
\label{lem:SE2}
Let $s> \frac12$. 
Then, 
there exists a finite constant $C=C(s)>0$ such that 
\begin{equation*}
\|\Tcal^{\textup{w}}_{|\Phi| > M}(v_1,v_2,v_3)\|_{H^{s}(\R)} 		
	\leq C M^{-\frac12}
	\prod^3_{j=1}  \|v_j\|_{H^s(\R)} \,, 
\end{equation*}
 for any $M\ge 1$. 
\end{lemma}
\begin{proof}
It is an immediate consequence of Lemma~\ref{locmodest} 
taking into account that the multiplier of the operator $\Tcal^{\textup{w}}_{|\Phi|>M}$ 
has an additional $\frac12$-power of $\jb{\Phi(\bxi)}$ in the denominator 
as compared to the multiplier of  $\Tcal_{\Phi}$ 
and that 
in the domain of integration we have $|\Phi(\bxi)|>M$. 
\end{proof}

\begin{definition}
\label{DEF:S} 
Let $J\ge 1$ and  $\TT\in  \TF(J)$.  
We define the $(2J+1)$-linear map $\SF^{\textup{w}}(\TT;\,\cdot\,)$ 
on space-time functions $v_j \in C(I; H^s(\R))$ ($1\le j\le 2J+1=|\TT^{\infty}|$) 
by the following rules. 

\begin{itemize}

\item[(i)] Replace the $j$th terminal node of $\TT$ 
by $v_j$, for all $j\in\{1,\ldots, 2J+1\}$.

\item[(ii)] 
For $j=J, J-1, \ldots, 1$, 
replace the $j$th root node $r^{(j)}$  
by  the  trilinear operator $\Tcal^{\textup{w}}_{|\Phi|>b_j N/2}$ 
whose  arguments are given by 
 the functions associated with its three children. 

\end{itemize}
\end{definition}

We have  the following immediate consequence of  Lemmata~\ref{lem:WE2} and \ref{lem:SE2}. 

\begin{corollary}
\label{cor:estSFw}
Let $s>\frac{1}{2}$, $\theta=\theta(s)=\min\{2s-1,\frac{1}{2}\}$, $J\ge 1$,  and $\TT\in \TF(J)$.    
Then, for any $1\le j\le 2J+1$ we have
\begin{equation*}
\big\|\SF^{\textup{w}}(\TT ; v_1,\ldots ,v_{2J+1}) \big\|_{H_x^s(\R)} \leq 
\frac{(2^{\theta}C)^J}{ b_1^{\theta} b_2^{\theta} \cdots b_{J-1}^{\theta} } N^{-\theta J} 
\|v_j\|_{H_x^{s-1}(\R)}
 \prod_{\substack{k=1\\ k\neq j}}^{2J+1} \|v_k\|_{H_x^s(\R)}   \,,
\end{equation*} 
where $C$ is the maximum between the two constants given by Lemmata~\ref{lem:WE2} and \ref{lem:SE2}.   
\end{corollary}
\begin{proof}
We  apply iteratively Lemma~\ref{lem:WE2} or Lemma~\ref{lem:SE2}. 
Let $a_j$ denote the $j$th terminal node of $\TT$.
Since $\TT$ is a tree of generation $J$, 
it has $J$ many root nodes $r^{(1)}, r^{(2)}, \ldots, r^{(J)}$, where $r^{(j)}\in\pi_j(\TT)$, $1\leq j\leq J$.  
Let $1\leq k\leq J$ such that the root node $r^{(k)}\in \pi_k(\TT)$ is the parent of the $j$th terminal node $a_j$. 
We recall (see Remark~\ref{REM:order}) that there exists the shortest path $P(r^{(1)}, r^{(k)})=r^{(k_1)}, r^{(k_2)},\ldots,r^{(k_\ell)}$ of root nodes from $r^{(1)}\equalscolon r^{(k_1)}$ to $r^{(k)}\equalscolon r^{(k_\ell)}$, $1=k_1<k_2<\ldots<k_\ell=k$. 

We prove the desired estimate by moving top-down on $\TT$ with a chronicle $\{\TT_j\}_{j=1}^J$. 
Starting with  $j=1$, if $a_j$ is a child of $r^{(1)}$, then we just apply Lemma~\ref{lem:WE2}. Otherwise, $T_1$ has one child (and only one) that belongs to $P(r^{(1)}, r^{(k)})$ which is $r^{(k_2)}\in\pi_{k_2}(\TT)$, $1<k_2\leq k$. 
So we use Lemma~\ref{lem:WE2}, placing the subtree with root node $r^{(k_2)}$ in the $H^{s-1}(\R)$-norm and the other two subtrees (possibly, it can be just one node) in the $H^s(\R)$-norm. 
In a similar manner, we continue to move down the path $r^{(k_2)}, \ldots,r^{(k_{\ell-1})}, r^{(k)}$ and each time we apply Lemma~\ref{lem:WE2} analogously. 
For any subtree of $\TT$ whose root node  does not belong to $\{r^{(1)}, r^{(k_2)}, \ldots,r^{(k_{\ell-1})}, r^{(k)}\}$, 
we use Lemma~\ref{lem:SE2} in chronological order. 
Notice that (modulo the constant $C$), the coefficient provided by the latter lemma is smaller than the one provided by the former. In the worst cases scenario 
(i.e. the tree is ``linear'' so that $k=J$, and $P(r^{(1)},r^{(k)})=r^{(1)}, r^{(2)},\ldots, r^{(J)}$),
we only apply Lemma~\ref{lem:WE2} to pick up the coefficient
$$C^J \Bigg(\prod_{j=1}^{J-1} \bigg(\frac{b_j N}{2}\bigg)^{-\theta }  \Bigg) = (2^\theta C)^J \Bigg(\prod_{j=1}^{J-1} b_j^{-\theta }  \Bigg) N^{-\theta J} \,,$$
with $b_j$ given by \eqref{defn:bJ}. 
\end{proof}

\subsection{Convergence to zero of the remainder term}
\label{subsect:4p1}

Here, 
we argue that for fixed $N>1$,  the remainder term 
$\Tcal_{\Tcal}^{(J+1)}(v)$ of \eqref{NFReqstepJ} converges to zero in the $H^{s-1}(\R)$-norm as $J\to\infty$.

\begin{lemma}
\label{lem:convtozero}
Let $s>\frac{1}{2}$ and $\theta=\theta(s)=\min\{2s-1,\frac{1}{2}\}$.  
Then, for $\Tcal_{\Tcal}^{(J+1)}(v)$ given by \eqref{defnTcalTcalJp1}, 
we have 
\begin{equation}
\label{decayofremainder}
\|\Tcal_{\Tcal}^{(J+1)}(v)\|_{H^{s-1}(\R)} \lesssim N^{- \theta J} \|v\|_{H^s(\R)}^{2J+3} \,.
\end{equation}
\end{lemma}
\begin{proof}
The formula \eqref{defnTcalTcalJp1} for $\Tcal_{\Tcal}^{(J+1)}(v)$ was obtained by replacing $\partial_t v$ with $\Tcal(v)$ in $\Tcal^{(J+1)}(v)$.
On the other hand, the same formula \eqref{defnTcalTcalJp1} can also obtained by replacing one $v$ in $\Tcal_0^{(J)}$ with $\Tcal(v)$.
More precisely, we can write 
\begin{equation}
\label{GeneralTcal0J+1}
\begin{aligned}
\Tcal_{\Tcal}^{(J+1)}(v) &= 
	\sum_{\TT\in\TF(J)} \sum_{j=1}^{2J+1} \int_{\bmxi\in\Xi_{\xi}(\TT)} 
 \ind_{F_J}  \Bigg( \prod_{j=1}^{J} \frac{e^{i {\mu}_{J}t } \xi^{(j)}_2}{\widetilde{\mu}_j}  \Bigg) 
 \bigg(\Tcal(v)(\xi_{a_k}) \prod_{\substack{a\in\TT^{\infty}\\ a\neq a_k}} v(\xi_a)\bigg)\\
&=:\sum_{\TT\in \TF(J)} \sum_{k=1}^{2J+1} \Tcal_{0}^{(J+1)}(\TT, a_k; \mathbf{v}_k) \,,
\end{aligned}
\end{equation}
where $a_j$ denotes the $k$th terminal node of  $\TT$,
and
for simplicity, we put 
$$\mathbf{v}_k=(v,\ldots,v,\underbrace{\Tcal(v)}_{k\text{th spot}},v,\ldots,v) \,.$$
We then have 
\begin{align}
\label{TcalTcalJp1sup}
\|\Tcal^{(J+1)}_{\Tcal} (v) \|_{H_x^{s-1}(\R)} 
& \leq  c_J(2J+1) \sup_{\TT\in \TF(J)}  
 \sup_{1\le k\le 2J+1}\| \Tcal_{0}^{(J+1)}(\TT, a_k; \mathbf{v}_k)\|_{H^{s-1}} \,.
\end{align}
Proceeding as in the proof of Lemma~\ref{lem:Tcal0Jp1}, 
we have $\frac12 |\wt \mu_j| < |\mu_j| < 2|\wt \mu_j|$ for $j=1, \ldots, J$ 
(due to the of integration being restricted to $F_J$)  and $|\mu_j| > \frac{1}{2}|\wt \mu_j| >\frac{b_{j-1}  N}{2}$. 
Therefore, we have 
\begin{align*}
\Tcal_{0}^{(J+1)}(\TT, a_j; \mathbf{v}_j) 
&\leq
\int_{\bmxi\in\Xi_{\xi}(\TT)} 
  \ind_{F_J} \bigg(\prod_{j=1}^{J} \frac{ |\xi^{(j)}_2|}{|\widetilde{\mu}_j|} \bigg) 
  \bigg(\Tcal(v)(\xi_{a_k}) \prod_{\substack{a\in\TT^{\infty}\\ a\neq a_k}} v(\xi_a)\bigg)  \\
&\leq 2^J \int_{\bmxi\in\Xi_{\xi}(\TT)} 
  \ind_{F_J} \bigg(\prod_{j=1}^{J} \frac{ |\xi^{(j)}_2|}{|{\mu}_j|} \bigg) 
  \bigg(\Tcal(v)(\xi_{a_k}) \prod_{\substack{a\in\TT^{\infty}\\ a\neq a_k}} v(\xi_a)\bigg) \\
&\leq  (2\sqrt{2})^J \SF^{\textup{w}}(\TT ; \mathbf{v}_k )
\end{align*}
With Corollary~\ref{cor:estSFw} and $\theta= \min\{2s-1,\frac{1}{2}\}$, 
we get  
\begin{equation}
\label{WeakTcal0J+1}
\|\Tcal_{0}^{(J+1)}(\TT, a_k; \mathbf{v}_k)  \|_{H^{s-1}_x(\R)} 
\le (4C)^J \bigg( \prod_{j=1}^J b_j\bigg)^{-\theta} N^{-\theta J} 
\|\Tcal(v)\|_{H_x^{s-1}(\R)} \|v\|^{2J}_{H_x^s(\R)}   
\end{equation}
for each $\TT\in\TF(J)$ and $1\leq j\leq 2J+1$.
Then, by \eqref{TcalTcalJp1sup} and Lemma~\ref{estTcalvHsm1} we get 
\begin{align*}
\|\Tcal^{(J+1)}_{\Tcal} (v) \|_{H_x^{s-1}(\R)} 
& \leq \frac{c_J(2J+1) (4C)^J}{b_1^{\theta} \cdots b_{J-1}^{\theta} } N^{-\theta J} \|v\|^{2J+3}_{H_x^s(\R)}
\end{align*}
The desired estimate \eqref{decayofremainder} follows by taking into account Remark~\ref{rmk:choicebetas}. 
\end{proof}

\section{Justification of the normal form reductions for rough solutions}
\label{sect:justif}

In each step of the infinite iteration in Section~\ref{sect2} 
we performed normal form reductions (NFR) which relied on two formal operations 
which obviously hold if $v$ is assumed to be a  smooth solution to \eqref{irvgDNLS}.   
Namely, 
(i) we applied the product rule 
when distributing the time derivative over products of several factors of $v$ 
(see e.g. \eqref{NFR1prod} below), 
and (ii) we switched the time derivative with  integrals in spatial frequencies 
(see e.g. \eqref{NFR1switch} below). 
In this section, we justify these operations for  a rough solution $v$ to \eqref{irvgDNLS}.

Let $s>\frac{1}{2}$, $\theta=\theta(s)=\min\{2s-1,\frac{1}{2}\}$,  and let $I$ be an interval containing $t=0$. 
Suppose that $v\in C(I; H^s(\R))$ is a solution to \eqref{irvgDNLS}, namely it satisfies (in the sense of distributions) the Duhamel formula
\begin{equation}
\label{DUHirvgDNLS}
v(t) = v_0 + \int_0^t \Qcal(v)(t')dt' + \int_0^t \Tcal(v)(t')dt' ,
\end{equation}
with $\Qcal$, $\Tcal$ as in \eqref{Qcal}, \eqref{Tcal}, respectively. 
By Lemma~\ref{lem:dtvest}, we have $v \in C^1_t\big(I;H_x^{s-1}(\R)\big)$. 
With $p$, $q\in (1,\infty)$ such that $\frac2p + \frac1q=1$ and $\frac12-\frac1q \le s-1$, 
by H\"{o}lder inequality and Sobolev embedding, we also have that 
\begin{align*}
\|v_1(\partial_x\overline{v_2})v_3\|_{L^1_x(\R)} 
 &\leq \|v_1\|_{L_x^p(\R)} \|\partial_x \overline{v_2}\|_{L^q_x(\R)} \|v_3\|_{L_x^p(\R)} \\
 &\lesssim \|v_1\|_{H_x^{s}(\R)} \|v_2\|_{H^s_x(\R)} \|v_3\|_{H_x^{s}(\R)}\, .
\end{align*}
Note that the condition $\frac12-\frac1p\le s$ is automatically satisfied. 
Therefore, we have
\begin{equation*}
\|\Tcal(v)\|_{H_x^{s-1}(\R)} + \|\Tcal(v)\|_{L_x^1(\R)}  
\lesssim \|v\|_{H^s_x(\R)}^3 \,.
\end{equation*}
Note that all of the above estimates hold uniformly in $t\in I$. 
For the quintic term in \eqref{DUHirvgDNLS}, we immediatelly have
\begin{equation*}
\|\Qcal(v)\|_{H_x^{s}(\R)} + \|\Qcal(v)\|_{L_x^1(\R)}  
\lesssim \|v\|_{H^s_x(\R)}^5 .
\end{equation*}
Moreover, by the Riemann-Lebesgue lemma, it follows that 
\begin{equation*}
\widehat{\Qcal(v)}\,,\, \widehat{\Tcal(v)} \in C_t(I; C_{\xi}(\R)) 
\end{equation*}
with 
\begin{align*}
&\|\widehat{\Tcal(v)}\|_{L^{\infty}_{\xi}(\R)} \lesssim \|v\|_{H^s_x(\R)}^3 \,, \\
& \|\widehat{\Qcal(v)}\|_{L^{\infty}_{\xi}(\R)}  \lesssim \|v\|_{H^s_x(\R)}^5 \,.
\end{align*}

By taking the Fourier transform of \eqref{DUHirvgDNLS}, 
by Fubini's theorem, 
we get
\begin{equation*}
\widehat{v}(t,\xi) = \widehat{v_0}(\xi) + \int_0^t \widehat{\Qcal(v)}(t',\xi)dt' 
  + \int_0^t \widehat{\Tcal(v)}(t',\xi)dt' . 
\end{equation*}
and by taking time derivative for fixed $\xi\in\R$, we have
\begin{equation*}
\partial_t \widehat{v}(t,\xi) = \widehat{\Qcal(v)}(t,\xi) + \widehat{\Tcal(v)}(t,\xi), 
\end{equation*}
for each $(t,\xi)\in I\times\R$. 
It follows that 
\begin{equation}
\label{vhatC1C}
\widehat{v} \in C^1_t\big(I;C_{\xi}(\R) \big).
\end{equation}

\subsection{Justification of the first step of NFR} 
Here, we carefully justify that $v$ is also a solution to \eqref{NF1gDNLS}, 
namely that  the Duhamel formula 
\begin{align*}
v(t) = v_0 & + \int_0^t \Qcal(v)(t')dt' + \int_0^t \Tcal_{\Tcal,1}^{(1)}(v)(t')dt' +  \Tcal_0^{(2)}(v)(t) - \Tcal_0^{(2)}(v)(0)  \\
&+  \int_0^t \Tcal_{\Qcal}^{(2)}(v)(t')dt'  + \int_0^t \Tcal_{\Tcal}^{(2)}(v)(t')dt' 
\end{align*}
is satisfied in the sense of distributions. 
Due to \eqref{vhatC1C}, 
it is immediate that the application of the product rule 
\begin{gather}
\label{NFR1prod}
\begin{split}
\partial_t\Big(\widehat{v}(t,\xi_1) \widehat{\partial_x v}(t,\xi_2) \widehat{v}(t,\xi_3)\Big)&=
 \big(\partial_t \widehat{v}(t,\xi_1)\big) \widehat{\partial_x v}(t,\xi_2) \widehat{v}(t,\xi_3)\\
 &\qquad + \widehat{v}(t,\xi_1) \big(\partial_t \widehat{\partial_x v}(t,\xi_2) \big) \widehat{v}(t,\xi_3)\\
 &\qquad + \widehat{v}(t,\xi_1) \widehat{\partial_x v}(t,\xi_2)  \big(\partial_t\widehat{v}(t,\xi_3)\big)
\end{split}
\end{gather}
is justified for all $t\in I$ and all  $\xi_1,\xi_2,\xi_3\in \R$.  

Next, we would like to justify the following: 
\begin{equation}
\label{NFR1switch}
\partial_t \Bigg[ \int_{\R^2} f(t,\xi,\xi_1,\xi_2) d\xi_1d\xi_2\Bigg] = 
	\int_{\R^2} \partial_t f (t,\xi,\xi_1,\xi_2)d\xi_1d\xi_2 \,,
\end{equation}
where the function $f:I\times\R^3\to \C$ is given by 
\begin{equation*}
f(t,\xi,\xi_1,\xi_2) =  \ind_{C_0} \frac{e^{i\Phi(\bar\xi)t}}{i\Phi(\bar\xi)} 
\widehat{v}(t,\xi_1) \overline{\widehat{\partial_x {v} }(t,\xi_2)} \widehat{v}(t,\xi-\xi_1+\xi_2) \,,
\end{equation*}
i.e. the integrand for $\Tcal_0^{(2)}(v)$ -- see \eqref{Tcal02}. 
We have that 
\begin{align*}
\partial_t f(t,\xi,\xi_1,\xi_2) 
&= \ind_{C_0} e^{i\Phi(\bar\xi)t} \widehat{v}(t,\xi_1) 
\overline{ \widehat{\partial_x {v} }(t,\xi_2) } \widehat{v}(t,\xi-\xi_1+\xi_2)\\
 &\quad + 
 \ind_{C_0} \frac{e^{i\Phi(\bar\xi)t}}{i\Phi(\bar\xi)}  \big(\partial_t \widehat{v}(t,\xi_1) \big)
 \overline{\widehat{\partial_x {v} }(t,\xi_2)} \widehat{v}(t,\xi-\xi_1+\xi_2) \\
 &\quad +\ind_{C_0} \frac{e^{i\Phi(\bar\xi)t}}{i\Phi(\bar\xi)}  \widehat{v}(t,\xi_1) 
 \big(\overline{\partial_t \widehat{\partial_x {v} }(t,\xi_2) }\big) \widehat{v}(t,\xi-\xi_1+\xi_2)\\
 &\quad + \ind_{C_0} \frac{e^{i\Phi(\bar\xi)t}}{i\Phi(\bar\xi)}  \widehat{v}(t,\xi_1) 
 \overline{\widehat{\partial_x {v} }(t,\xi_2)} \big(\partial_t\widehat{v}(t,\xi-\xi_1+\xi_2)\big)\\
 &=: g_1(t,\xi,\xi_1,\xi_2) +g_2(t,\xi,\xi_1,\xi_2) + g_3(t,\xi,\xi_1,\xi_2) + g_4(t,\xi,\xi_1,\xi_2)
\end{align*}
By omitting any complex constants of modulus one, we can write
\begin{align*}
\int_{\R^2} f(t,\xi,\xi_1,\xi_2) d\xi_1d\xi_2
&= \F\big[\Tcal_0^{(2)}(v) \big](t,\xi)\\
\int_{\R^2} g_1(t,\xi,\xi_1,\xi_2) d\xi_1d\xi_2 
 &= \F\big[\Tcal_2(v) \big](t,\xi) \\
\int_{\R^2} g_2(t,\xi,\xi_1,\xi_2) d\xi_1d\xi_2 
 & = \F\big[ \Tcal_0^{(2)}(\partial_t v, v,v) \big](t,\xi) \\
\int_{\R^2} g_3(t,\xi,\xi_1,\xi_2) d\xi_1d\xi_2 
 &= \F\big[  \Tcal_0^{(2)}(v,\partial_t v, v) \big](t,\xi)  \\
\int_{\R^2}g_4(t,\xi,\xi_1,\xi_2) d\xi_1d\xi_2 
 &= \F\big[  \Tcal_0^{(2)}(v,v,\partial_t v) \big](t,\xi) \,,
\end{align*}
where $\Tcal_0^{(2)}(v)$, $\Tcal_2(v)$ are given by \eqref{Tcal02}, respectively.
 Furthermore, we set  
\begin{align*}
&F:=  \Tcal_0^{(2)}(v)\,,\\
&G_1:= \Tcal_2(v)\,,\ 
G_2:= \Tcal_0^{(2)}(\partial_t v, v,v) \,,\ 
G_3:=  \Tcal_0^{(2)}(v,\partial_t v, v)\,,\ 
G_4:= \Tcal_0^{(2)}(v,v,\partial_t v) \,,
\end{align*}
$g:=g_1+g_2+g_3+g_4$, 
and $G:=G_1+G_2+G_3+G_4$. 
Thus \eqref{NFR1switch} follows once we show that 
$\partial_t  F= G$ 
holds  in the sense of distributions.

By Lemma~\ref{lem:WE2}, we deduce\footnote{For the continuity in time of $F$, 
one uses the multilinear version of the estimate provided by Lemma~\ref{lem:WE2}.} 
that $F\in C(I; H^{s-1}(\R))$ with 
\begin{equation}
\label{NFR1est1}
\|F(t)\|_{H_x^{s-1}} \lesssim N^{-\theta} \| v\|_{H_x^s}^3\,.
\end{equation}
Similarly, we have  that  $G\in C(I;H^{s-1}(\R))$ since 
%
%
by Lemma~\ref{estTcalvHsm1} and Lemma~\ref{lem:WE2}, we have
\begin{gather}
\label{NFR1est2}
\begin{split}
\|G(t)\|_{H_x^{s-1}}
&\leq  \big\|\Tcal_2(v)\big\|_{H_x^{s-1}} 
	+ \big\| \Tcal^{\textup{w}}_{|\Phi|>N}(\partial_t v, v,v) \big\|_{H_x^{s-1}}
	+ \big\| \Tcal^{\textup{w}}_{|\Phi|>N}(v,\partial_t v, v) \big\|_{H_x^{s-1}}\\
&\qquad\qquad\qquad\quad\ \	+\big\| \Tcal^{\textup{w}}_{|\Phi|>N}(v,v,\partial_t v) \big\|_{H_x^{s-1}}\\ 
&\lesssim \| v\|_{H_x^s}^3 +  \|\partial_t v\|_{H_x^{s-1}} \|v\|^2_{H_x^s}\\
&\lesssim  \| v\|_{H_x^s}^3+ \|v\|^5_{H_x^s} + \|v\|^7_{H_x^s}\,,
\end{split}
\end{gather}
where in the last step we applied Lemma~\ref{lem:dtvest}. 


Now fix $t\in I$ 
and let $\varphi\in \mathcal{S}(\R)$.  
By the Plancherel formula, we have 
\begin{align*}
\int_{\R} F(t,x) \varphi(x) dx &= \int_{\R^3} f(t,\xi,\xi_1,\xi_2) \widehat{\varphi}(\xi) d\xi_1d\xi_2 d\xi\,,\\
\int_{\R} G(t,x) \varphi(x) dx &= \int_{\R^3} g(t,\xi,\xi_1,\xi_2) \widehat{\varphi}(\xi) d\xi_1d\xi_2 d\xi\,.
\end{align*}
By appealing to the Fourier lattice property of the Sobolev spaces $H^{s-1}$, $H^{1-s}$,  
to the Riemann-Lebesgue lemma
and by using \eqref{NFR1est2}, 
we have
\begin{equation*}
 | g(t,\xi,\xi_1,\xi_2) \widehat{\varphi}(\xi)| 
\lesssim  \|G\|_{H_x^{s-1}} \big\| \F^{-1}\big[|\widehat{\varphi}|^{\frac12}\big]\big\|_{H_x^{1-s}} 
 	|\widehat{\varphi}(\xi)|^{\frac12}
 \lesssim_{\|v\|_{C(I;H^s(\R))}} |\widehat{\varphi}(\xi)|^{\frac12}\,.
\end{equation*}
and thus the dominated convergence theorem implies: 
\begin{equation*}
\partial_t \int_{\R} F(t,x) \varphi(x)dx = \int_{\R} G(t,x) \varphi(x)dx \,.
\end{equation*}

\subsection{Justification of the $J$th step of NFR} 
In justifying the first  step of NFR, the main ingredients\footnote{Whenever we apply the product rule, we appeal to \eqref{vhatC1C}.} are the estimates \eqref{NFR1est1} and \eqref{NFR1est2}. 
For a generic step $J$,  
we briefly show how to derive the corresponding estimates. 
To this end, fix $\TT\in \TF(J)$ and note that for \eqref{NFRJdiffparts}, we have used the following: 
\begin{equation}
\label{NFRJswitch}
\partial_t \Bigg[ \int_{\bmxi\in\Xi_{\xi}(\TT)} f(t,\xi,\bmxi) \Bigg] = 
	\int_{\bmxi\in\Xi_{\xi}(\TT)} \partial_t f (t,\xi,\bmxi)  \,,
\end{equation}
where the function $f:I\times\Xi(\TT)\to \C$ is given by 
\begin{equation*}
f(t,\xi,\bmxi) =   \ind_{F_J}
  \bigg(\prod_{j=1}^{J}  \frac{e^{i {\mu}_{j}t } \xi^{(j)}_2}{ \widetilde{\mu}_j} \bigg) \bigg(\prod_{a\in\TT^{\infty}} v(\xi_a)\bigg)\,,
\end{equation*}
i.e. the integrand for $\Tcal_0^{(J+1)}(v)$ -- see \eqref{defnTcal0Jp1}. 
Note that 
\begin{align*}
\int_{\bmxi\in\Xi_{\xi}(\TT)} f (t,\xi,\bmxi) &= \F\big[\Tcal_0^{(J+1)}(\TT;v)\big](t,\xi) =: \F\big[F\big](t,\xi)\,,\\
\int_{\bmxi\in\Xi_{\xi}(\TT)} \partial_t f (t,\xi,\bmxi) 
  &= \F\bigg[\Tcal_{\Tcal, 2}^{(J)}(\TT;v) + \sum_{k=1}^{2J+1} \Tcal_0^{(J+1)}(\TT,a_k; \widetilde{\mathbf{v}}_k)\bigg](t,\xi)=: \F\big[G\big](t,\xi)\,,
\end{align*}
where $\Tcal_0^{(J+1)}(\TT,a_k; \widetilde{\mathbf{v}}_k)$ in the summation above is defined by replacing $\mathbf{v}_k$ in \eqref{GeneralTcal0J+1} by
$$\widetilde{\mathbf{v}}_k=(v,\ldots,v,\underbrace{\partial_t v}_{k\text{th spot}},v,\ldots,v),$$
and $a_k$ is the $k$th terminal node  of $\TT\in\Tfrak(J)$.

Similarly to \eqref{WeakTcal0J+1} in the proof of Lemma~\label{lem:convtozero} with  Corollaries~\ref{lem:dtvest},
we have
\begin{align*}
\big\|\Tcal_0^{(J+1)}(\TT; v)\|_{H_x^{s-1}(\R)} &\lesssim  N^{-\theta J} \|v\|^{2J+1}_{H_x^s(\R)} 
\,,\\
\big\|\Tcal_0^{(J+1)}(\TT,a_k; \widetilde{\mathbf{v}}_{k})\|_{H_x^{s-1}(\R)} &\lesssim  N^{-\theta J} \|v\|^{2J+3}_{H_x^s(\R)}
 \big( 1+ \|v\|^{2}_{H_x^s(\R)}  \big)\,,\quad k=1,\ldots,2J+1.
\end{align*}
Also, similarly to the proof of Lemma~\ref{lem:TcalTcal1Jp1}, 
with  Corollary~\ref{cor:estSFw} and Lemma~\ref{estTcalvHsm1}, 
we get 
\begin{align*}
\big\|\Tcal_{\Tcal, 2}^{(J)}(\TT;v)\|_{H_x^{s-1}(\R)} &\lesssim  N^{-\theta(J-1)} \|v\|^{2J+1}_{H_x^s(\R)}\,.
\end{align*}
It follows that $F,$ $G \in C(I; H^{s-1}(\R))$ with 
\begin{align}
\label{NFRJest1}
\|F\|_{H_x^{s-1}(\R)} &\lesssim \|v\|_{H_x^s(\R)}^{2J+1} \,,\\
\label{NFRJest2}
\|G\|_{H_x^{s-1}(\R)} &\lesssim \|v\|_{H_x^s(\R)}^{2J+1} + \|v\|_{H_x^s(\R)}^{2J+3} + \|v\|_{H_x^s(\R)}^{2J+5}\,.
\end{align}
Similarly to the previous subsection, by appealing to the dominated convergence theorem and 
\eqref{NFRJest1}, \eqref{NFRJest2} one justifies \eqref{NFRJswitch}.

\smallskip

Together with Lemma~\ref{lem:convtozero}, we conclude that
the Duhamel formula of the equation \eqref{limiteq} 
is satisfied in the sense of distributions, provided that $v\in C(I; H_x^s(\R))$ is a solution to \eqref{irvgDNLS}.

\section{Proof of Theorem~\ref{mainthm}} 
\label{sect:pfmainthm}

First, we summarilly go over the fixed point argument for \eqref{limiteq} 
with prescribed initial data $v(0)=v_0\in H^s(\R)$, 
$s>\frac{1}{2}$. Integrating the limit equation \eqref{limiteq} in time,  
we obtain the folllowing Duhamel formulation: 
\begin{gather}
\begin{split}
\label{NFEq:intf}
v(t)=&\, v_0 + \int_0^t\Qcal(v)(t')dt' + \sum_{j=2}^{\infty} \Big(\Tcal_0^{(j)}(v)(t) - \Tcal_0^{(j)}(v)(0)\Big)
	+ \sum_{j=2}^{\infty}  \int_0^t\Tcal^{(j)}_{\Qcal}(v)(t')dt'\\
	&\qquad + \sum_{j=1}^{\infty} \int_0^t \Tcal^{(j)}_{\Tcal,1}(v)(t')dt' \,.
\end{split}
\end{gather}
Let us denote the right-hand side of \eqref{NFEq:intf} by $\Gamma(v)$, 
and for simplicity we write $C_TH^s$ instead of $C([-T,T]; H^s(\R))$. 

Having the estimates of Section~\ref{Sect:strongests}, 
one can show that $\Gamma$ is a contraction on the ball 
$\mathcal{B}_T:=\{v \in C_TH^s : \|v\|_{C_TH^s}\le 2\|v_0\|_{H^s} \}$, 
provided  that $T>0$ and $N>1$ are appropriately chosen.  
Indeed, 
we set $R:=2\|v_0\|_{H^s}$,
and thus 
by Lemmata~\ref{locmodest}, \ref{lem:Tcal0Jp1}, \ref{lem:TcalQcalJp1}, 
 and \ref{lem:TcalTcal1Jp1}, we get
\begin{align*}
\|\Gamma(v)\|_{C_TH^s} &\leq \frac12 R + T R^5 
  + c \sum_{j=2}^{\infty} N^{-\frac{1}{2} (j-1)} R^{2(j-1)+1} +  c  T \sum_{j=2}^{\infty} N^{-\frac{1}{2} (j-1)} R^{2(j-1)+5}\\
  &\qquad + cT \sum_{j=1}^{\infty} N^{-\frac{1}{2} (j-2)} R^{2(j-1)+3}\\
 &\leq  \frac12 R + T R^5 + c \frac{N^{-\frac{1}{2}} R^3}{1-N^{-\frac{1}{2}}R^2} + cT \frac{N^{-\frac{1}{2}} R^7}{1-N^{-\frac12}R^2} \\
 &\qquad+ cTN^{\frac12} R^3 +cTR^5 + cT \frac{N^{-\frac{1}{2}}R^7}{1-N^{-\frac12}R^2}\\
 &\leq  \frac12 R + (1+c)T R^5  + 2c(1+2TR^4)N^{-\frac12}R^3 +c T N^{\frac12}R^3.
\end{align*}
for some $c=c(s)>0$, when $N \geq 4R^4$ so that $(1-N^{-\frac{1}{2}}R^2)^{-1}\leq 2$. 
First, we choose $T_1=T_1(R)>0$ such that $(1+c)T_1R^4 \leq \frac16$,  
then we choose $N=N(R)\ge 1+ 4R^4$ such that $2c(1+ 2T_1R^2)N^{-\frac{1}{2}}R^2 \leq \frac16$, 
and finally we choose $T=\min\big\{T_1, \frac16 (cN^{\frac12} R^2)^{-1}\big\}$.

By possibly choosing smaller $T$ and bigger $N$ and  by using the difference estimates of 
Lemmata~\ref{lem:Tcal0Jp1}, \ref{lem:TcalQcalJp1}, \ref{locmodest}, and \ref{lem:TcalTcal1Jp1},  
the contraction property of $\Gamma$ follows analogously. 
Therefore, by the contraction mapping principle, for given $v_0\in H^s(\R)$, 
there exists a unique $v\in C_TH^s$ 
satisfying \eqref{NFEq:intf}. Moreover, $\|v\|_{C_TH^s}\lesssim \|v_0\|_{H^s}$.



Now let us consider  two solutions $u_1, u_2\in C_TH^s$ of DNLS. 
By Lemma~\ref{gaugetrLip}, $w_1, w_2\in C_TH^s$ 
and 
$$\|u_1-u_2\|_{C_TH^s}\lesssim \|w_1-w_2\|_{C_TH^s} = \|v_1-v_2\|_{C_TH^s}\,,$$
where $v_j(t):=S(-t)w_j(t)$, $t\in [-T,T]$, are solutions to \eqref{irvgDNLS}.  
Then, by the arguments of Section~\ref{sect:justif}, $v_1, v_2$ 
are solutions of the normal form equation \eqref{limiteq} derived in Section~\ref{sect2}. 
Similarly to the above lines of reasoning,  
we deduce
$$\|v_1-v_2\|_{C_TH^s} = \|\Gamma(v_1) - \Gamma(v_2)\|_{C_TH^s} 
\lesssim \|v_1(0)-v_2(0)\|_{H^s}=\|u_1(0)-u_2(0)\|_{H^s}$$
and thus  any two solutions $u_1, u_2\in C_TH^s$ started from the same initial data 
must coincide on the time interval $[-T,T]$. 
By appealing to the time translation symmetry of  DNLS, 
we conclude that any initial data $u_0\in H^s(\R)$ determines a unique solution to DNLS 
which is continuous in time with values in $H^s(\R)$.

\appendix
\section{Notation: indexing by ordered trees}
\label{appdxA}

We include here the notation and terminology used in \cite[Section~3.1]{KOY} regarding the cubic NLS equation on the real line. 

\begin{definition} \rm
Given a partially ordered set $\TT$ with partial order $\leq$, we say that $b\in\TT$ with $b\leq a$ and $b\neq a$ is a child of $a\in\TT$, if $b\leq c\leq a$ implies either $c=a$ or $c=b$. If the latter condition holds, we also say that $a$ is the parent of $b$.
\end{definition}

As in \cite{ChristPowerSeries, Oh}, 
the trees refer to a particular subclass of ternary trees.

\begin{definition} \rm
A ternary tree $\TT$ is a finite partially ordered set satisfying the following properties:

\begin{itemize}
\item[] Let $a_1$, $a_2$, $a_3$, $a_4\in\TT$. If $a_4\leq a_2\leq a_1$ and $a_4\leq a_3\leq a_1$, then we have $a_2\leq a_3$ or $a_3\leq a_2$.

\item[] A node $a\in\TT$ is called terminal, if it has no child. A non-terminal node $a\in\TT$ is a node with exactly three children denoted by $a_1$, $a_2$ and $a_3$.\footnote{Note
that the order of children plays an important role in our discussion.
We refer to $a_j$ as the $j$th child of a non-terminal node $a \in \TT$.
In terms of the planar graphical representation of a tree, 
we set  the $j$th node from the left as the $j$th child $a_j$ of $a \in \TT$. 
}

\item[]  There exists a maximal element $r\in\TT$ (called the root node) such that $a\leq r$ for all $a\in\TT$. We assume that the root node is non-terminal.

\item[]  $\TT$ consists of the disjoint union of $\TT^0$ and $\TT^\infty$, where $\TT^0$ and $\TT^\infty$ denote the collection of parental (non-terminal) nodes and terminal nodes, respectively.
\end{itemize}
\end{definition}

Note that the number $|\TT|$ of nodes in a tree $\TT$ is $3j+1$ for some $j\in\mathbb{N}$, 
where $|\TT^0|=j$ and $|\TT^\infty|=2j+1$. 
Next, we recall the notion of ordered trees introduced in \cite{GKO}. 
Roughly speaking, an ordered tree ``remembers how it grew''.

\begin{definition}\label{DEF:tree2}\rm
We say that a sequence $\{\TT_j\}_{j=1}^J$ is a chronicle of $J$ generations, if
\begin{itemize}
\item[]  $\TT_j$ has $j$ parental nodes for each $j=1, \ldots, J$,
\item[]   $\TT_{j+1}$ is obtained by changing one of the terminal nodes in $\TT_j$,
denoted by $p^{(j)}$, 
 into a non-terminal node (with three children), $j=1, \ldots, J-1$.
\end{itemize}

\noi
Given a chronicle $\{\TT_j\}_{j=1}^J$ of $J$ generations, 
we refer to $\TT_J$ as an {\it ordered tree of the $J$th generation}. 
We use $\TF(J)$ to denote the collection of the ordered trees of the $J$th generation.

\end{definition}

Note that the cardinality of $\TF(J)$ is given by
	\begin{equation}\label{cj}
	|\TF(J)|=1\cdot 3\cdot 5\cdot\cdots\cdot (2J-1)=: c_J
	\end{equation}

\begin{remark}\rm

Given two ordered trees $\TT_J$ and $\wt{\TT}_J$
of the $J$th generation, 
it may happen that $\TT_J = \wt{\TT}_J$ as trees (namely as graphs) 
while $\TT_J \ne \wt{\TT}_J$ as ordered trees according to Definition \ref{DEF:tree2}.
Henceforth, when we refer to an ordered tree $\TT_J$ of  the $J$th generation, 
it is understood that there is an underlying chronicle $\{ \TT_j\}_{j = 1}^J$.

\end{remark}

\begin{definition} \label{DEF:tree3}\rm
(i) Given an ordered tree $\TT_J \in \TF(J)$ 
with a chronicle $\{\TT_j\}_{j = 1}^J$, 
we define a ``projection'' $\pi_j$, $j = 1, \dots, J$,  from $\TT_J$
to subtrees in $\TT_J$ of one generation 
by setting
\begin{itemize}
\item[]  $\pi_1(\TT_J) = \TT_1$, 
\item[]  
$\pi_j(\TT_J)$ to be the tree formed by the three terminal nodes in $\TT_j \setminus \TT_{j-1}$
and its parent,  $j = 2, \dots, J$.   
Intuitively speaking, 
$\pi_j(\TT_J)$ is the tree added in transforming $\TT_{j-1}$ into $\TT_j$.

\end{itemize}

\noi
We use $r^{(j)}$ to denote the root node of $\pi_j(\TT_J)$
and refer to it as the {\it $j$th root node}.
By definition, we have
\begin{align}
r^{(j)} = p^{(j-1)}.
\label{tree3}
\end{align}

\noi
Note that
$p^{(j-1)}$ is not necessarily a node in 
$\pi_{j-1}(\TT_J)$.

\smallskip
\noi
(ii) Given $j \in \{1, \dots, J-1\}$, 
$p^{(j)}$  appears as a terminal node of  $\pi_{k}(\TT)$
for exactly one $k \in \{1, 2\dots, j-1\}$.
In particular, 
$p^{(j)}$ is the $\l$th child of the $k$th root note $r^{(k)}$
for some $\l\in \{1, 2, 3\}$.
We define {\it the order of $p^{(j)}$}, denoted by 
$\#p^{(j)}$, 
to be this number $\l\in \{1, 2, 3\}$.

\smallskip
\noi
(iii) 
We  define the {\it essential terminal nodes} $\pi_j^\infty(\TT_J)$ of the $j$th generation by 
setting
\[\pi_j^\infty(\TT_J): = \pi_j(\TT_J)^\infty \cap\TT_J^\infty
= (\TT_j \setminus \TT_{j-1})\cap \TT_J^\infty.\]

\noi
By definition, 
$\pi_j^\infty(\TT_J)$ may be empty.
Note that $\{ \pi_j^\infty(\TT_J)\}_{j = 1}^J$ forms 
a partition of $\TT_J^\infty$.
\end{definition}

We record the following simple observation.

\begin{remark}\label{REM:order}\rm
Let $\TT \in \TF(J)$ be an ordered tree. 
Then, for each fixed  $j = 2, \dots, J$, 
there exists a path\footnote{A path is  a sequence of nodes $a_1, a_2, \dots, a_K$
such that 
$a_k$ and $a_{k+1}$ are adjacent.}
$a_1, a_2, \dots, a_K$,
starting at the root node $r = r^{(1)}$
and ending at 
the $j$th root node $r^{(j)}$
such that $a_k \ne r^{(\l)}$
for any $k = 1, \dots, K$
and $\l \geq  j+1$.
Namely, we can move from $r^{(1)}$
to $r^{(j)}$
without hitting a root node of a higher generation.

More concretely, given $r^{(j)}$, 
we know that it  appears as a terminal node of  $\pi_{j_1}(\TT)$
for exactly one $j_1 \in \{1, 2\dots, j-1\}$.
Similarly,  $r^{(j_1)}$ appears as a terminal node of  $\pi_{j_2}(\TT)$
for exactly one $j_2 \in \{1, 2\dots, j_1-1\}$.
We can iterate this process, which must terminate in a finite number of steps
with $j_k = 1$.
This generates the shortest path
$r^{(j_k)}, r^{(j_{k-1})}, \dots, r^{(j_1)}, r^{(j)}$
from $r^{(1)}$
to $r^{(j)}$
and we denote it by $P(r^{(1)}, r^{(j)})$.
Similarly, given $a\in \TT\setminus\{r^{(1)}\}$, 
one can easily construct the shortest path from $r^{(1)}$ to $a$
since $a$ is a terminal node of $\pi_k(\TT)$ for some $k$.
We denote this shortest path by 
$P(r^{(1)}, a)$.

\end{remark}

Given an ordered tree, 
we need to 
consider all possible frequency assignments
to nodes that are ``consistent''.

\begin{definition}\label{DEF:index}
\rm
Given an ordered tree $\TT \in \TF(J)$,  
we define an {\it index function} $\pmb{\xi}:\TT\to\mathbb{R}$ such that
\begin{align}
\xi_a=\xi_{a_1}-\xi_{a_2}+\xi_{a_3}
\label{index}
\end{align}

\noi
for $a\in\TT^0$, where $a_1$, $a_2$, and $a_3$ denote the children of $a$.
Here,  we identified $\pmb{\xi}:\TT\to\R$ with $\{\xi_a\}_{a\in\TT}\in\mathbb{R}^\TT$. 
We use $\Xi(\TT)\subset\R^{\TT}$ to denote the collection of such index functions $\pmb{\xi}$. 
Also, the collection of index functions $\bmxi\in \Xi(\TT)$ 
with fixed frequency $\xi\in\R$ at the root node of $\TT$ 
is denoted by $\Xi_{\xi}(\TT)$
\end{definition}

\begin{remark}\rm 
 If we associate functions $v_a = v_a(\xi_a)$ to each node $a \in \TT$, 
then the relation~\eqref{index} implies that $v_a = v_{a_1}* \cj{v_{a_2}}*v_{a_3}$.

\end{remark}

Given an ordered tree 
$\TT_J \in \TF(J)$  with a chronicle $\{ \TT_j\}_{j = 1}^J$ 
and associated index functions $\pmb{\xi} \in \Xi(\TT_J)$,
 we use superscripts to keep track of  ``generations'' of frequencies.

Consider $\TT_1$ of the first generation.
We define the first generation of frequencies by
\[\big(\xi^{(1)}, \xi^{(1)}_1, \xi^{(1)}_2, \xi^{(1)}_3\big) :=(\xi_r, \xi_{r_1}, \xi_{r_2}, \xi_{r_3}),\]

\noi
where $r_j$ denotes the three children of the root node $r$.

In general, the ordered tree $\TT_j$ 
of the $j$th generation is obtained from $\TT_{j-1}$ by
changing one of its terminal nodes $a  \in \TT^\infty_{j-1}$
into a non-terminal node.
Then, we define
the $j$th generation of frequencies by
\[\big(\xi^{(j)}, \xi^{(j)}_1, \xi^{(j)}_2, \xi^{(j)}_3\big) :=(\xi_a, \xi_{a_1}, \xi_{a_2}, \xi_{a_3}),\]

\noi
\noi
where $a_j$ denotes the three children of the  node  $a  \in \TT^\infty_{j-1}$.
Note that the parent node $a$ is nothing but the $j$th root node $r^{(j)}$
defined in Definition \ref{DEF:tree3}.

\medskip

Our main analytical tool is the localized modulation estimate
of Lemma~\ref{locmodest}.
Hence, it is important to keep track of the modulation 
for frequencies in  each generation.
We use $\mu_j$  to denote the corresponding modulation function  introduced at the $j$th generation.
Namely, we set
\begin{align*}
\mu_j & = \mu_j \big(\xi^{(j)}, \xi^{(j)}_1, \xi^{(j)}_2, \xi^{(j)}_3\big)
:= \big(\xi^{(j)}\big)^2 - \big(\xi_1^{(j)}\big)^2 + \big(\xi_2^{(j)}\big)^2- \big(\xi_3^{(j)}\big)^2 \notag \\
& = 2\big(\xi_2^{(j)} - \xi_1^{(j)}\big) \big(\xi_2^{(j)} - \xi_3^{(j)}\big)
= 2\big(\xi^{(j)} - \xi_1^{(j)}\big) \big(\xi^{(j)} - \xi_3^{(j)}\big), 
\end{align*}
where the last two equalities hold in view of \eqref{index}. 
We also use the following short-hand notation:
\begin{equation*}
\wt \mu_j := \sum_{k = 1}^j \mu_k.
\end{equation*}
Given $\xi\in\R$ and $\TT\in\mathfrak{T}(J)$, we use a short-hand notation for iterated integrals of the form
\begin{equation*}
 \int_{\bmxi\in\Xi_{\xi}(\TT)} \ [\,\cdot\,]\ 
 := 
 \underbrace{\int_{\R^2} \ldots \int_{\R^2}}_{J\text{ times}}\ [\,\cdot\,]\  
 d\xi^{(J)}_1d\xi^{(J)}_2\ldots  d\xi^{(1)}_1d\xi^{(1)}_2.
\end{equation*}

\subsection*{Acknowledgements}
The authors would like to thank their advisors Tadahiro Oh  and Soonsik Kwon 
 for the discussions about this work. 
Since part of this paper was completed while RM visited 
the Department of Mathematics at Kyoto University, Japan in February 2018, 
he would like to thank Professor  Yoshio~Tsutsumi for the hospitality and support.  
  The authors are grateful to Justin Forlano for his  proofreading and discussions about this work. 

RM further acknowledges support 
from the Maxwell Institute Graduate School in Analysis and its Applications, 
a Centre for Doctoral Training funded by 
the UK Engineering and Physical Sciences Research Council (grant EP/L016508/01), 
the Scottish Funding Council, Heriot-Watt University and the University of Edinburgh.
Most of the material of this article 
was included in Chapter~4 of the first author's PhD thesis \cite{MosincatThesis}.

\bigskip

\end{document}